\documentclass[12pt]{amsart}
\usepackage{amsmath}
\usepackage{amsxtra}
\usepackage{amstext}
\usepackage{amssymb}
\usepackage{amsthm}
\usepackage{latexsym}
\usepackage{dsfont} 
\usepackage{verbatim}
\usepackage{tabls}
\usepackage{graphicx}
\usepackage{rotating}
\usepackage{caption}
\theoremstyle{plain}
\newtheorem{thm}{Theorem}[section]

\newtheorem{lem}[thm]{Lemma}
\newtheorem{prop}[thm]{Proposition}

\theoremstyle{definition}
\newtheorem{defn}[thm]{Definition}

\newtheorem{rem}[thm]{Remark}

\newtheorem{cla}[thm]{Claim}

\numberwithin{equation}{section}
\def\supp{\operatorname{supp}}
\def\Lip{\operatorname{Lip}}
\def\eps{\varepsilon}
\def\kap{\varkappa}
\def\Cau{\mathcal{C}}
\def\Comp{\mathbb{C}}

\def\XXint#1#2#3{{\setbox0=\hbox{$#1{#2#3}{\int}$}
     \vcenter{\hbox{$#2#3$}}\kern-.5\wd0}}

\begin{document}

\title[Reflectionless measures and rectifiability]
{Reflectionless measures and the Mattila-Melnikov-Verdera uniform rectifiability theorem}

\author[B. Jaye]
{Benjamin Jaye}
\address{Department of Mathematics,
Kent State University,
Kent, OH 44240, USA}
\email{bjaye@kent.edu}

\author[F. Nazarov]
{Fedor Nazarov}
\address{Department of Mathematics,
Kent State University,
Kent, OH 44240, USA}
\email{nazarov@math.kent.edu}

\date{\today}

\begin{abstract}The aim of these notes is to provide a new proof of the  Mattila-Melnikov-Verdera theorem on the uniform rectifiability of an Ahlfors-David regular measure whose associated Cauchy transform operator is bounded.  They are based on lectures given by the second author in the analysis seminars at Kent State University and Tel-Aviv University.\end{abstract}

\maketitle

\section{Introduction}


The purpose of these notes is to provide a new proof of Mattila, Melnikov, and Verdera's theorem.  The exposition is self-contained, relying only on a knowledge of basic real analysis.

\begin{thm}\label{thm1}\cite{MMV} An Ahlfors-David regular measure $\mu$ whose associated Cauchy transform operator is bounded in $L^2(\mu)$ is uniformly rectifiable.

\end{thm}

The precise statement of this theorem is given in Section \ref{statementsec}.  The scheme employed to prove Theorem \ref{thm1} in these notes is quite different from that in \cite{MMV}, and relies upon a characterization of \textit{reflectionless measures}.  In this regard, one may compare the proof to that of Mattila's theorem \cite{Mat95b}: \emph{Suppose that $\mu$ is a  finite Borel measure satisfying $\liminf_{r\rightarrow 0} \frac{\mu(B(z,r))}{r} \in (0,\infty)$ for $\mu$-a.e. $z\in \Comp$.  If the Cauchy transform of $\mu$ exists $\mu$-a.e. in the sense of principal value,  then $\mu$ is rectifiable}.  Mattila's proof of this theorem uses a characterization of \textit{symmetric measures}, the reader may consult Chapter 14 of the book \cite{Mat95} for more information.

Subsequently, Mattila's theorem was generalized to the case of singular integrals in higher dimensions by Mattila and Preiss in \cite{MP95}.   To find the analogous generalization of the proof we carry out here would answer a longstanding problem of David and Semmes \cite{DS}.  Very recently, Nazarov, Tolsa, and Volberg \cite{NTV12} completed the solution of the problem of David and Semmes in the case of singular integral operators of co-dimension 1.  They proved that if $\mu$ is a $d$-dimensional Ahlfors-David regular measure in $\mathbb{R}^{d+1}$, then the boundedness of the $d$-dimensional Riesz transform in $L^2(\mu)$ implies that one of the criteria for uniform rectifiability given in \cite{DS} is satisfied.  See \cite{NTV12} for more details, and for further references and history about this problem.

Throughout this paper, we shall only consider Ahlfors-David regular measures.  For closely related results without this assumption, see the paper of L\'{e}ger \cite{Leg}.

\section{Notation}

We shall adopt the following notation:

\begin{itemize}
\item  $B(z,r)$ denotes the open disc centred at $z$ with radius $r>0$.

\item For a square $Q$, we set $z_Q$ to be the centre of $Q$, and $\ell(Q)$ to denote the side-length of $Q$.

\item We shall denote by $\mathcal{D}$ the standard lattice of dyadic squares in the complex plane.  A dyadic square is any square of the form $[k2^{j}, (k+1)2^{j})\times [\ell 2^{j}, (\ell+1)2^{j})$ for $j, k$ and $\ell$ in $\mathbb{Z}$.

\item We define the Lipschitz norm of a function $f$ by $$\|f\|_{\Lip} = \sup_{z, \xi\in \mathbb{C}, z\neq \xi} \frac{|f(z) - f(\xi)|}{|z-\xi|}.$$

\item We denote by $\Lip_0(\mathbb{C})$ the space of compactly supported functions with finite Lipschitz norm.  The continuous functions with compact support are denoted by $C_0(\mathbb{C})$.

\item For $f:\mathbb{C}\rightarrow\mathbb{C}$, we set $\|f\|_{\infty} = \sup_{z\in \mathbb{C}}|f(z)|.$  In particular, note that we are taking the \textit{pointwise everywhere} supremum here.

\item The closure of a set $E$ is denoted by $\overline{E}$

\item  The support of a measure $\mu$ is denoted by $\supp(\mu)$.

\item For a line $L$, we write the one dimensional Hausdorff measure restricted to $L$ by $\mathcal{H}^1_{L}$.  If $L=\mathbb{R}$, we instead write $m_1$.

\item We will denote by $C$ and $c$ various positive absolute constants.  These constants may change from line to line within an intermediate argument.  The constant $C$ is thought of as large (at the very least greater than $1$), while $c$ is thought of as small (certainly smaller than $1$).  We shall usually make any dependence of a constant on a parameter explicit, unless it is clear from the context what the dependencies are.
\end{itemize}

\section{The precise statement of Theorem \ref{thm1}}\label{statementsec}

\subsection{The Cauchy transform of a measure $\mu$}  Let $K(z)=\tfrac{1}{z}$ for $z\in \mathbb{C}\backslash \{0\}$.  For a measure $\nu$, the \textit{Cauchy transform} of  $\nu$ is formally defined by 
$$\Cau(\nu)(z) = \int_{\Comp}K(z-\xi) d\nu(\xi),
 \text{ for }z\in \mathbb{C}.
$$
In general, the singularity in the kernel is too strong to expect the integral to converge absolutely on $\supp(\nu)$.  It is therefore usual to introduce a regularized Cauchy kernel.  For $\delta>0$, define
$$K_{\delta}(z) = \frac{\bar{z}}{\max(\delta, |z|)^2}.
$$
Then the $\delta$-regularized Cauchy transform of $\nu$ is defined by
$$\Cau_{\delta}(\nu)(z) =  \int_{\Comp}K_{\delta}(z-\xi) d\nu(\xi), \text{ for }z\in \mathbb{C}.
$$

Before we continue, let us introduce a very natural condition to place upon $\mu$.  A measure $\mu$ is called $C_0$-\emph{nice} if $\mu(B(z,r))\leq C_0 r$ for any disc $B(z,r)\subset \Comp$.

If $\mu$ is a $C_0$-nice measure, then for any $f\in L^2(\mu)$ and $z\in \Comp$, we have that $\Cau_{\mu, \delta}(f)(z):=\Cau_{\delta}(f\mu)(z)$ is bounded in absolute value in terms of $\delta$, $C_0$, and $\|f\|_{L^2(\mu)}$.  To see this, we shall need an elementary tail estimate, which we shall refer to quite frequently in what follows:

\begin{lem}\label{tailest}  Suppose $\mu$ is $C_0$-nice measure.  For every $\eps>0$ and $r>0$, we have
$$\int\limits_{\mathbb{C}\backslash B(0,r)} \frac{1}{|\xi|^{1+\eps}}d\mu(\xi)\leq \frac{C_0(1+\eps)}{\eps} r^{-\eps}.
$$
\end{lem}

The proof of Lemma \ref{tailest} is a standard exercise, and is left to the reader.  With this lemma in hand, we return to our claim that $C_{\mu, \delta}(f)$ is bounded.  First apply the Cauchy-Schwarz inequality to estimate $$|\Cau_{\mu, \delta}(f)(z)|\leq \Bigl(\int_{\mathbb{C}}|K_{\delta}(z-\xi)|^2d\mu(\xi)\Bigl)^{1/2}\|f\|_{L^2(\mu)}.$$
But now $\int_{\mathbb{C}}|K_{\delta}(z-\xi)|^2d\mu(\xi) \leq \int_{B(z,\delta)}\frac{|\xi-z|^2}{\delta^4}d\mu(\xi) + \int_{\mathbb{C}\backslash B(z,\delta)}\frac{1}{|\xi-z|^2}d\mu(\xi).
$
The first term on the right hand side of this inequality is at most $\tfrac{\mu(B(z,\delta))}{\delta^2}\leq \tfrac{C_0}{\delta}$, and the second term is no greater than $\tfrac{2C_0}{\delta}$ by Lemma \ref{tailest}.  We therefore see that $|\Cau_{\mu, \delta}(f)(z)|\leq \bigl(\tfrac{3C_0}{\delta}\bigl)^{1/2}\|f\|_{L^2(\mu)}$.  In particular, we have  $\Cau_{\mu, \delta}(f)(z)\in L^2_{\text{loc}}(\mu)$.

One conclusion of this discussion is that for any nice measure $\mu$, it makes sense to ask if $C_{\mu,\delta}$ is a bounded operator from $L^2(\mu)$ to $L^2(\mu)$.

\begin{defn} We say that $\mu$ is $C_0$-\emph{good} if it is $C_0$-nice and
$$\sup_{\delta>0}\|\Cau_{\mu, \delta}\|_{L^2(\mu)\rightarrow L^2(\mu)}\leq C_0.
$$
By definition, the Cauchy transform operator associated with $\mu$ is bounded in $L^2(\mu)$ if $\mu$ is good.
\end{defn}

The two dimensional Lebesgue measure restricted to the unit disc is good.  However, this measure is not supported on a $1$-rectifiable set and so such measures should be ruled out in a statement such as Theorem \ref{thm1}.  To this end, we shall deal with Ahlfors-David (AD) regular measures.

\begin{defn}\label{ADreg}  A nice measure $\mu$ is called AD-regular, with regularity constant $c_0>0$, if $\mu(B(z,r))\geq c_0 r$ for any disc $B(z,r)\subset \mathbb{C}$ with $z\in \supp(\mu)$.
\end{defn}



\subsection{Uniform rectifiability}  A set $E\subset \mathbb{C}$ is called \emph{uniformly rectifiable} if there exists $M>0$ such that for any dyadic square $Q\in \mathcal{D}$, there exists a Lipschitz mapping $F:[0,1]\rightarrow \mathbb{C}$ with $\|F\|_{\Lip}\leq M\ell(Q)$ and $E\cap Q \subset F([0,1])$.

We can altnernatively say that $E$ is uniformly rectifiable if there exists $M>0$ such that for any dyadic square $Q\in \mathcal{D}$, there is a rectifiable curve containing $E\cap Q$ of length no greater than $M\ell(Q)$.

A measure $\mu$ is uniformly rectifiable if the set $E=\supp(\mu)$ is uniformly rectifiable.

We may now restate Theorem \ref{thm1} in a more precise way.

\begin{thm}  A good AD-regular measure $\mu$ is uniformly rectifiable.
\end{thm}


\section{Making sense of the Cauchy transform on $\supp(\mu)$} \label{weaklims}
The definition of a good measure does not immediately provide us with a workable definition of the Cauchy transform on the support of $\mu$.  In this section, we rectify this matter by defining an operator $\Cau_{\mu}$ as a weak limit of the operators $\Cau_{\mu, \delta}$ as $\delta\rightarrow 0$.  This idea goes back to Mattila and Verdera \cite{MV}.  We fix a $C_0$-good measure $\mu$.

Note that if $f \in \Lip_0(\mathbb{C})$, then $f$ is bounded in absolute value by $\|f\|_{\Lip}\cdot \text{diam}(\supp(f))$.

Fix $f,g\in \Lip_0(\mathbb{C})$.  Then for any $\delta>0$, we may write
$$\langle \Cau_{\mu,\delta}(f),g\rangle_{\mu} = \frac{1}{2}\iint\limits_{\mathbb{C}\times\mathbb{C}} K_{\delta}(z-\xi)\bigl[f(\xi)g(z)-f(z)g(\xi)\bigl]d\mu(z)d\mu(\xi).
$$
Let $H(z,\xi) = \tfrac{1}{2}\bigl[f(\xi)g(z)-f(z)g(\xi)\bigl]$.  It will be useful to denote by $I_{\delta}(f,g)$ the expression
$$I_{\delta}(f,g) =I_{\delta, \mu}(f,g)= \iint\limits_{\mathbb{C}\times\mathbb{C}} K_{\delta}(z-\xi) H(z,\xi)d\mu(\xi)d\mu(z).
$$
Now, note that if $S = \supp(f)\cup\supp(g)$, it is clear that $\supp(H)\subset S\times S$.

In addition $H$ is Lipschitz in $\mathbb{C}^2$ with Lipschitz norm no greater than $\tfrac{1}{\sqrt{2}}\bigl(\|f\|_{\infty}\|g\|_{\Lip}+\|g\|_{{\infty}}\|f\|_{\Lip}\bigl)$.  To see this, first observe that
$|H(z,\xi) - H(\omega, \xi)|\leq \tfrac{1}{2}  \bigl(\|f\|_{\infty}\|g\|_{\Lip}+\|g\|_{\infty}\|f\|_{\Lip}\bigl)|z-\omega| $, whenever $z, \omega, \xi\in \mathbb{C}$.  By using this inequality twice, we see that $$|H(z_1,z_2) - H(\xi_1, \xi_2)|\leq \tfrac{1}{2}\bigl(\|f\|_{\infty}\|g\|_{\Lip}+\|g\|_{\infty}\|f\|_{\Lip}\bigl)[|z_1-\xi_1|+|z_2-\xi_2|], $$ and the claim follows since $|z_1-\xi_1|+|z_2-\xi_2|\leq \sqrt{2}\sqrt{|z_1-\xi_1|^2+|z_2-\xi_2|^2}$.
Since $H(z,z)=0$, this Lipschitz bound immediately yields $$|H(z,\xi)|\leq \tfrac{1}{\sqrt{2}}\bigl(\|f\|_{\infty}\|g\|_{\Lip}+\|g\|_{\infty}\|f\|_{\Lip}\bigl)|z-\xi| \text{ for any }z\neq \xi.$$

As a result of this bound on the absolute value of $H$, there exists a constant $C(f,g)>0$ such that
$$|K_{\delta}(z-\xi)||H(z,\xi)|\leq C(f,g)\chi_{S\times S}(z,\xi).$$
On the other hand, since $\mu$ is a nice measure, the set $\{(z,\xi)\in \mathbb{C}\times\mathbb{C}: \, z=\xi\}$ is $\mu\times\mu$ null, and so for $\mu\times\mu$ almost every $(z, \xi)$, the limit as $\delta\rightarrow0$ of $K_{\delta}(z-\xi)$ is equal to $K(z-\xi)$.  As a result, the Dominated Convergence Theorem applies to yield
$$\lim_{\delta\rightarrow 0}I_{\delta}(f,g) = \iint\limits_{\mathbb{C}\times\mathbb{C}} K(z-\xi)H(z,\xi)d\mu(z)d\mu(\xi).
$$
This limit will be denoted by $I(f,g)=I_{\mu}(f,g)$.  Moreover, there is a quantitative estimate on the speed of convergence:
$$|I(f,g) - I_{\delta}(f,g)|\leq \iint\limits_{\substack{ (z,\xi)\in S\times S:\\ |z-\xi|<\delta}}C(f,g) d\mu(z)d\mu(\xi) \leq C(f,g)\delta\mu(S).
$$
Since $\mu$ is $C_0$-nice, $\mu(S)$ can be bounded in terms of the diameters of the supports of $f$ and $g$, and we see that  $|I(f,g) - I_{\delta}(f,g)|\leq C(f,g)\delta$.

We have now justified the existence of an operator $\Cau_{\mu}$ acting from the space of compactly supported Lipschitz functions to its dual with respect to the standard pairing $\langle f,g\rangle_{\mu}=\int_{\mathbb{C}}fgd\mu$.  

Since $\mu$ is $C_0$-good, for any $\delta>0$ we have
\begin{equation}\label{Idelt}
|I_{\delta}(f,g)|\leq  C_0\|f\|_{L^2(\mu)}\|g\|_{L^2(\mu)},\text{ for any } f,g \in L^2(\mu),
\end{equation}
and this inequality allows us to extend the definition of $I(f,g)$ to the case when $f$ and $g$ are $L^2(\mu)$ functions.  To do this,  we first pick functions $f$ and $g$ in $L^2(\mu)$.   Let $\eps>0$.  Using the density of $\Lip_0(\Comp)$ in $L^2(\mu)$, we write $f=f_1+f_2$ and $g=g_1+g_2$, where $f_1$ and $g_1$ are compactly supported Lipschitz functions, and the norms of $f_2$ and $g_2$ in $L^2(\mu)$ are as small as we wish (say, less than $\eps$).  We know that $I_{\delta}(f_1,g_1)\rightarrow I(f_1,g_1)$ as $\delta\rightarrow 0$.  Consequently, for each $\eps>0$, $I_{\delta}(f,g)$ can be written as a sum of two terms, the first of which (namely $I_{\delta}(f_1,g_1)$) has a finite limit, and the second term (which is $I_{\delta}(f_1,g_2)+I_{\delta}(f_2,g_1)+I_{\delta}(f_2,g_2)$) has absolute value no greater than $ C_0\eps(3\eps+ \|f\|_{L^2(\mu)}+\|g\|_{L^2(\mu)})$.  It follows that the limit as $\delta\rightarrow 0$ of $I_{\delta}(f,g)$ exists.  We define this limit to be $I(f,g)=I_{\mu}(f,g)$.

From (\ref{Idelt}), we see that $|I(f,g)|\leq C_0\|f\|_{L^2(\mu)}\|g\|_{L^2(\mu)}$.  
Therefore we may apply the Riesz-Fisher theorem to deduce the existence of a (unique) bounded linear operator $\Cau_{\mu}:L^2(\mu)\rightarrow L^2(\mu)$ such that
$$\langle \Cau_{\mu}(f), g\rangle_{\mu} = I(f,g) \text{ for all }f,g\in L^2(\mu).
$$

Having defined an operator $\Cau_{\mu}$ for any good measure $\mu$, we now want to see what weak continuity properties this operator has.

\begin{defn}  We say that the sequence $\mu_k$ tends to $\mu$ weakly if, for any $\varphi\in C_0(\mathbb{C})$,
$$\int_{\mathbb{C}}\varphi \,d\mu_k \rightarrow \int_{\mathbb{C}}\varphi \,d\mu \text{ as } k\rightarrow \infty.
$$
\end{defn}

We now recall a standard weak compactness result, which can be found in Chapter 1 of \cite{Mat95} (or any other book in real analysis).

\begin{lem}\label{bookconv}  Let $\{\mu_k\}_k$ be a sequence of measures.  Suppose that for each compact set $E\subset \mathbb{C}$, $\sup_k \mu_k(E)<\infty$.  Then there exists a subsequence  $\{\mu_{k_j}\}_{k_j}$ and a measure $\mu$ such that $\mu_{k_j}$ converges to $\mu$ weakly.
\end{lem}

An immediate consequence of this lemma is that any sequence $\{\mu_k\}_k$ of $C_0$-nice measures has a subsequence that converges weakly to a measure $\mu$.  The next lemma shows that the various regularity properties of measures that we consider are inherited by weak limits.

\begin{lem}\label{wlnicegood}  Suppose that $\mu_k$ converges to $\mu$ weakly.  If each measure $\mu_k$ is $C_0$-good with $AD$ regularity constant $c_0$, then the limit measure $\mu$ is also
$C_0$-good with $AD$ regularity constant $c_0$.
\end{lem}

\begin{proof} We shall first check that $\mu$ is AD regular.  Let $x\in \text{supp}(\mu)$, $r>0$, and choose $\eps\in (0,r/2)$.  Consider a smooth non-negative function $f$, supported in the disc $B(x,\eps)$, with $f\equiv 1$ on $B(x, \tfrac{\eps}{2})$.  Then $\int_{\mathbb{C}} fd\mu>0$.  Hence, for all sufficiently large $k$, $\int_{\mathbb{C}} f d\mu_k>0$.  For all such $k$, $B(x,\eps)\cap \supp(\mu_k)\neq \varnothing$, and so there exists $x_k\in B(x,\eps)$ satisfying $\mu_k(B(x_k, r-2\eps))\geq c_0 (r-2\eps)$.    As a result, $\mu_k(B(x,r-\eps))>c_0(r-2\eps)$.

Now let $\varphi\in C_0(\Comp)$ be nonnegative and supported in $B(x,r)$, satisfying $\|\varphi\|_{\infty}\leq 1$ and  $\varphi \equiv 1$ on $B(x, r-\eps)$.  Then
$$\mu(B(x,r))\geq \int_{\mathbb{C}} \varphi d\mu = \lim_{k\rightarrow \infty}\int_{\mathbb{C}} \varphi d\mu_k \geq c_0(r-2\eps).
$$
Letting $\eps\rightarrow 0$, we arrive at the desired AD regularity.  The property that $\mu$ is $C_0$-nice is easier and left to the reader (it also follows from standard lower-semicontinuity properties of the weak limit).

It remains to show that $\mu$ is $C_0$-good.  Fix $f,g\in \Lip_0(\mathbb{C})$ and define $H$ and $S$ as before.  Note that $K_{\delta}(z-\xi)H(z,\xi)$ is a Lipschitz function in $\mathbb{C}^2$, and has support contained in $S\times S$.  Let $U\supset S$ be an open set with $\mu(U)\leq \mu(S)+1$.  The (complex valued) Stone-Weierstrass theorem for a locally compact space tells us that that the algebra of finite linear combinations of functions in $C_0(U)\times C_0(U)$ is dense in $C_0(U\times U)$ (with respect to the uniform norm in $\Comp^2$).  Let $\eps>0$.  There are functions $\varphi_1, \dots, \varphi_n$ and $\psi_1,\dots, \psi_n$, all belonging to $C_0(U)$, such that  $|K_{\delta}(z-\xi)H(z,\xi)-\sum_{j=1}^n\varphi(z)\psi(\xi)|<\eps$ for any $(z,\xi)\in U\times U$.  For each $j=1,\dots,n$, we have
$$\lim_{k\rightarrow \infty} \iint\limits_{\mathbb{C}\times\mathbb{C}} \varphi_j(z)\psi_j(\xi)d\mu_k(z)d\mu_k(\xi) = \iint\limits_{\mathbb{C}\times\mathbb{C}} \varphi_j(z)\psi_j(\xi)d\mu(z)d\mu(\xi).
$$
It therefore follows that $$\limsup_{k\rightarrow \infty}| I_{\delta, \mu_k}(f,g) - I_{\delta, \mu}(f,g)|\leq \eps (\limsup_{k\rightarrow \infty}\mu_k(U)^2+\mu(U)^2).$$  On the other hand, $\mu_k$ is $C_0$ nice, and so $\mu_k(U)\leq C(f,g)$.  Since $\eps>0$ was arbitrary, we conclude that  $I_{\delta, \mu_k}(f,g) \rightarrow I_{\delta, \mu}(f,g)$ as $k\rightarrow \infty$.

As a result of this convergence, we have that
$$|I_{\mu, \delta}(f,g)|\leq C_0\liminf_{k\rightarrow \infty}\bigl(\|f\|_{L^2(\mu_k)}\|g\|_{L^2(\mu_k)}\bigl).
$$
But since both $|f|^2$ and $|g|^2$ are in $C_0(\mathbb{C})$, the right hand side of this inequality equals $\|f\|_{L^2(\mu)}\|g\|_{L^2(\mu)}$.  We now wish to appeal to the density of $\Lip_0(\mathbb{C})$ in $L^2(\mu)$ to extend this inequality to all $f,g\in L^2(\mu)$.  Let $R>0$.  As $\mu$ is $C_0$-nice, we saw in Section \ref{statementsec} that $\Cau_{\mu,\delta}: L^2(\mu)\rightarrow L^2(B(0,R),\mu)$.   But then, since the space of Lipschitz function compactly supported in $B(0,R)$ is dense in $L^2(B(0,R),\mu)$, we conclude that $||\Cau_{\mu,\delta}||_{L^2(\mu)\rightarrow L^2(B(0,R),\mu)}\leq C_0$. Finally, taking the limit as $R\rightarrow \infty$, the monotone convergence theorem guarantees that $\|\Cau_{\mu, \delta}\|_{L^2(\mu)\rightarrow L^2(\mu)}\leq C_0$, and hence $\mu$ is $C_0$-good.\end{proof}

The proof of the next lemma is left as an exercise.

\begin{lem}\label{ADsupp}  Suppose that $\mu_k$ is a sequence of $c_0$ $AD$-regular measures converging weakly to a measure $\mu$.  If $z_k\in \supp(\mu_k)$ with $z_k\rightarrow z$ as $k\rightarrow \infty$, then $z\in \supp(\mu)$.

\end{lem}

Our last task is to check that the bilinear form $I_{\mu_k}$ has nice weak convergence properties. For this, let $f,g\in \Lip_0(\mathbb{C})$.  For $\delta>0$,  we write
\begin{equation}\begin{split}\nonumber|I_{\mu_k}(f,g) - I_{\mu}(f,g)| \leq &
\, |I_{\mu_k}(f,g) -I_{\delta, \mu_k}(f,g)|+|I_{\delta, \mu_k}(f,g)-I_{\delta, \mu}(f,g)| \\
&+ |I_{\delta, \mu}(f,g)- I_{\mu}(f,g)|.
\end{split}\end{equation}
The first and third terms are bounded by $C(f,g)\delta$.  The second term converges to $0$ as $k\rightarrow \infty$.  Therefore
$$\limsup_{k\rightarrow \infty}|I_{\mu_k}(f,g) - I_{\mu}(f,g)| \leq C(f,g)\delta.
$$
But $\delta>0$ was arbitrary, and so $I_{\mu_k}(f,g)$ converges to  $I_{\mu}(f,g)$ as $k\rightarrow \infty$.

\section{Riesz Systems}\label{rieszsyssec}

Throughout this section we fix a $C_0$-nice measure $\mu$.


A system of functions $\psi_Q$ $(Q\in \mathcal{D})$ is called a $C$-Riesz system if $\psi_Q \in L^2(\mu)$ for each $Q$, and
\begin{equation}\label{Rieszsys}\Bigl|\Bigl| \sum_{Q\in \mathcal{D}} a_Q \psi_Q \Bigl|\Bigl|^2_{L^2(\mu)} \leq C\sum_{Q\in \mathcal{D}} |a_Q|^2,
\end{equation}
for every sequence $\{a_Q\}_{Q\in \mathcal{D}}$.  By a simple duality argument, we see that if $\psi_Q$ is a $C$-Riesz system, then
$$\sum_{Q\in \mathcal{D}} \bigl|\langle f, \psi_Q\rangle_{\mu}\bigl|^2 \leq C \|f\|^2_{L^2(\mu)}, \text{ for any }f\in L^2(\mu).
$$

Suppose now that with each square $Q\in \mathcal{D}$, we associate a set $\Psi_Q$ of $L^2(\mu)$ functions.  We say that $\Psi_Q$ $(Q\in \mathcal{D})$ is a $C$-Riesz family if, for any choice of functions $\psi_Q \in \Psi_Q$, the system $\psi_Q$ forms a $C$-Riesz system.

We now introduce a particularly useful Riesz family.  Suppose that $\mu$ is a $C_0$-nice measure.  Fix $A>1$, and define
\begin{equation}\begin{split}\nonumber\Psi^{\mu}_{Q,A}\! \!= \!\!\Bigl\{ \psi: \supp(\psi)\!\subset B(z_Q, A\ell(Q)),& \,\|\psi\|_{\Lip}\leq \ell(Q)^{-3/2},\int_{\mathbb{C}}\! \psi \,d\mu=0\Bigl\}.
\end{split}\end{equation}

\begin{lem}  For any $A>1$, $\Psi^{\mu}_{Q,A}$ is a $C$-Riesz family, with constant $C=C(C_0, A)$.
\end{lem}

\begin{proof}For each $Q\in \mathcal{D}$, pick a function $\psi_Q \in \Psi^{\mu}_{Q,A}$.  Then we have 
$$\|\psi_Q\|_{\infty} \leq \|\psi_{Q}\|_{\Lip}\cdot \text{diam}(\supp(\psi_Q))\leq \ell(Q)^{-3/2} \cdot 2A\ell(Q)\leq  CA\ell(Q)^{-1/2},
$$
and $$\|\psi_Q\|_{L^2(\mu)}^2\leq ||\psi_Q||_{L^{\infty}}^2\mu(B(z_Q,A\ell(Q)))\leq CA^3.$$

Now, if $Q', Q''\in \mathcal{D}$ with $\ell(Q')\leq \ell(Q'')$, then $\langle \psi_{Q'}, \psi_{Q''}\rangle_{\mu} =0$  provided that $B(z_{Q'}, A\ell(Q'))\cap B(z_{Q''}, A\ell(Q''))=\varnothing.$
If $B(z_{Q'}, A\ell(Q'))$ intersects  $B(z_{Q''}, A\ell(Q''))$, we instead have the bound
$$|\langle \psi_{Q'}, \psi_{Q''}\rangle_{\mu}|\leq CA^3 \Bigl(\frac{\ell(Q')}{\ell(Q'')}\Bigl)^{3/2}.
$$
Indeed, note that $\|\psi_{Q'}\|_{L^1(\mu)}\leq ||\psi_{Q'}||_{L^{\infty}}\mu(B(z_{Q'}, A\ell(Q')))\leq CA^2\ell(Q')^{1/2}$, while the oscillation of $\psi_{Q''}$ on the set $B(z_{Q'}, A\ell(Q'))$ (which contains the support of $\psi_{Q'}$) is no greater than $\tfrac{A\ell(Q')}{\ell(Q'')^{3/2}}$.  By multiplying these two estimates we arrive at the desired bound on the absolute value of the inner product.

Consider a sequence $\{a_Q\}_Q \in \ell^2(\mathcal{D})$.  Then
$$\Bigl|\Bigl| \sum_{Q\in \mathcal{D}} a_Q \psi_Q \Bigl|\Bigl|^2_{L^2(\mu)} \leq 2\sum_{\substack{Q', Q''\in \mathcal{D} :\\ \ell(Q')\leq \ell(Q'')}}|a_{Q'}||a_{Q''}||\langle \psi_{Q'}, \psi_{Q''}\rangle_{\mu}|.
$$
 Inserting our bounds on the inner products into this sum, we see that we need to bound the sum
$$CA^3\sum_{\substack{\ell(Q')\leq \ell(Q''), \\ B(z_{Q'}, A\ell(Q'))\cap B(z_{Q''}, A\ell(Q''))\neq \varnothing}}\!\!\!|a_{Q'}\|a_{Q''}|\Bigl(\frac{\ell(Q')}{\ell(Q'')}\Bigl)^{3/2}.
$$
(Since all sums involving squares will be taken over the lattice $\mathcal{D}$, we will not write this explicitly from now on.)  Using Cauchy's inequality, we estimate
 $$|a_{Q'}\|a_{Q''}|\Bigl(\frac{\ell(Q')}{\ell(Q'')}\Bigl)^{3/2}\leq \frac{|a_{Q'}|^2}{2}\Bigl(\frac{\ell(Q')}{\ell(Q'')}\Bigl)^{1/2}+ \frac{|a_{Q''}|^2}{2}\Bigl(\frac{\ell(Q')}{\ell(Q'')}\Bigl)^{5/2}.
 $$
It therefore suffices to estimate two double sums:
 $$
I =  \sum_{Q'} |a_{Q'}|^2 \sum_{\substack{Q'':\ell(Q')\leq \ell(Q''), \\ B(z_{Q'}, A\ell(Q'))\cap B(z_{Q''}, A\ell(Q''))\neq \varnothing}} \Bigl(\frac{\ell(Q')}{\ell(Q'')}\Bigl)^{1/2},
 $$
 and
 $$
 II = \sum_{Q''}|a_{Q''}|^2 \sum_{\substack{Q':\ell(Q')\leq \ell(Q''), \\ B(z_{Q'}, A\ell(Q'))\cap B(z_{Q''}, A\ell(Q''))\neq \varnothing}} \Bigl(\frac{\ell(Q')}{\ell(Q'')}\Bigl)^{5/2}.
 $$
 For each dyadic length $\ell$ greater than $\ell(Q')$, there are at most $CA^2$ squares $Q''$ of side length $\ell$ for which $B(z_{Q''}, A\ell)$ has non-empty intersection with $B(z_{Q'}, A\ell(Q'))$.  Hence
 $$I\leq \sum_{Q'} |a_{Q'}|^2 \sum_{k\geq 0} CA^2 2^{-k/2}\leq CA^2 \sum_Q|a_Q|^2.$$
 Concerning $II$, all the relevant squares $Q'$ in the inner sum are contained in the disc $B(z_{Q''}, 3A\ell(Q''))$.  Therefore, at scale $\ell$ there are at most $CA^2  \bigl(\frac{\ell(Q'')}{\ell}\bigl)^{2}$ squares $Q'$ of side length $\ell$ that can contribute to the inner sum.  As a result,
 $$II \leq CA^2 \sum_{Q''}|a_{Q''}|^2\sum_{k\geq 0} 2^{-k/2}\leq CA^2\sum_Q|a_Q|^2.
 $$
 Combining our bounds, we see that the $\Psi^{\mu}_{Q,A}$ is a Riesz family, with Riesz constant $C(C_0)A^5$.\end{proof}

 \section{Bad squares and uniform rectifiability}

In this section we identify a local property of the support of a measure, which ensures that the measure is uniformly rectifiable.  The mathematics in this section is largely due to David, Jones, and Semmes, see \cite{DS} Chapter 2.1, and is simpler than Jones'  geometric Traveling Salesman theory \cite{Jon}, which was used in \cite{MMV}.

Fix a $C_0$-nice measure $\mu$, which is AD-regular with regularity constant $c_0$.   Set $E=\supp(\mu)$.

\subsection{The construction of a Lipschitz mapping}

We will begin by constructing a certain graph.  For our purposes, a \textit{graph} $\Gamma=(\mathcal{N}, \mathcal{E})$ is a set of points $\mathcal{N}$ (the vertices), endowed with a collection of line segments $\mathcal{E}$ (the edges) where each segment has its end-points at vertices.  A \textit{connected component} of the graph is a maximal subset of vertices that can be connected through the edges.  For example, the graph depicted in Figure 1 below has two connected components.   

The \textit{distance between connected components} of a graph is measured as the distance between the relevant sets of vertices.  Therefore, the distance between the components of the graph depicted in Figure 1 is the distance between the vertices labeled $p$ and $q$.

\begin{figure}[h]\label{comppic}
\centering
 \includegraphics[height = 30mm]{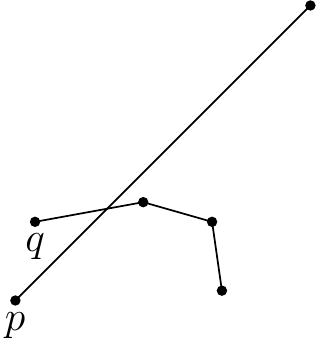}
\caption[A graph.]{An example of a graph consisting of two connected components.}
 \end{figure}


\begin{defn}For a graph $\Gamma=(\mathcal{N}, \mathcal{E})$, and a square $Q$, we define $\Gamma_Q$ to be the subgraph with vertex set $\mathcal{N}\cap 7Q$, endowed with the edges from $\mathcal{E}$ connecting those vertices in $7Q$.\end{defn}

Let $\tau \in (0,1)$.  Fix $P\in \mathcal{D}$ (this square is to be considered as the viewing window in the definition of uniform rectifiability).  Choose a (small) dyadic fraction $\ell_0$ with $\ell_0<\ell(P)$.

We shall construct a graph adapted to $P$ inductively.  Set $\mathcal{N}$ to be a maximal $\tau \ell_0$ separated subset of $E$.  Note that $\mathcal{N}$ forms a $\tau\ell_0$ net of $E$.

\textbf{\emph{The base step.}}  For each square $Q\in \mathcal{D}$ with $\ell(Q)=\ell_0$ and $3Q\cap \mathcal{N}\neq\varnothing$, fix a point which lies in $3Q\cap\mathcal{N}$.  Then join together every point of $\mathcal{N} \cap 3Q$ to this fixed point by line segments, as illustrated in the figure below.

In $3Q$, there are at most $C\tau^{-2}$ points of $\mathcal{N}$, and so the total length of the line segments in $3Q$ is $C\tau^{-2}\ell_0$.

\begin{figure}[h]\label{basesteppic}
\centering
 \includegraphics{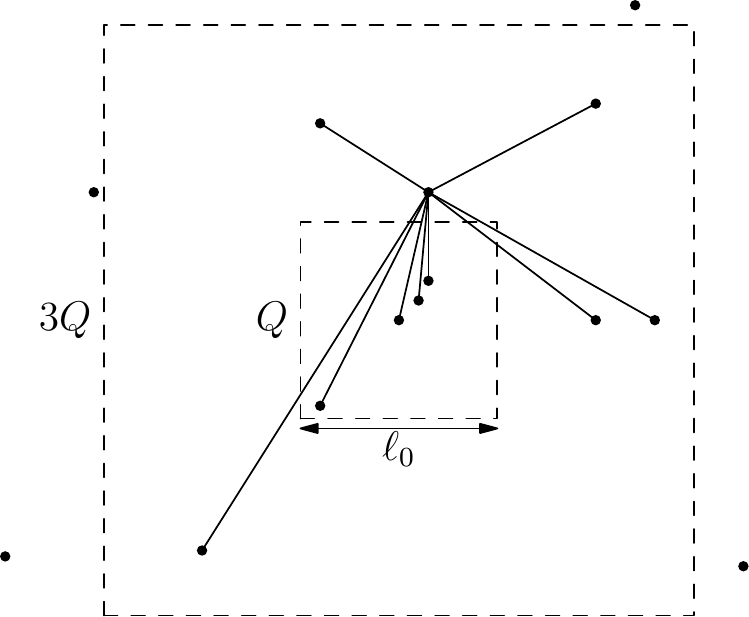}
\caption[A graph.]{The base step in the construction applied in $3Q$.}
 \end{figure}

 We thereby form the graph $\Gamma_{\ell_0}(\ell_0)$ comprised of the vertex set $\mathcal{N}$, and the set of line segments $\mathcal{E}_{\ell_0}(\ell_0)$ obtained by carrying out the above procedure for all squares $Q\in \mathcal{D}$ with $\ell(Q)=\ell_0$.


 This is the base step of the construction.


\textbf{\emph{The inductive step.}}    Let $\ell$ be a dyadic fraction no smaller than $\ell_0$.  Suppose that we have constructed the graph $\Gamma_{\ell_0}(\ell) = (\mathcal{N}, \mathcal{E}_{\ell_0}(\ell))$.

The graph $\Gamma_{\ell_0}(2\ell)$ is set to be the pair $(\mathcal{N}, \mathcal{E}_{\ell_0}(2\ell))$, where $\mathcal{E}_{\ell_0}(2\ell)$ is obtained by taking the union of $\mathcal{E}_{\ell_0}(\ell)$ with the collection of line segments obtained by performing the following algorithm:

\textit{For every square $Q\in \mathcal{D}$ with $\ell(Q)=2\ell$, consider the graph $\Gamma = (\Gamma_{\ell_0}(\ell))_Q$.
 If $\Gamma$ has at least two components that intersect $3Q$, then for each such component, choose a vertex that lies in its intersection with $\mathcal{N}\cap 3Q$.  Fix a point in $3Q\cap \mathcal{N}$, and join each of the chosen points to this fixed point with an edge}.

\begin{figure}[h]\label{comppic}
\centering
 \includegraphics{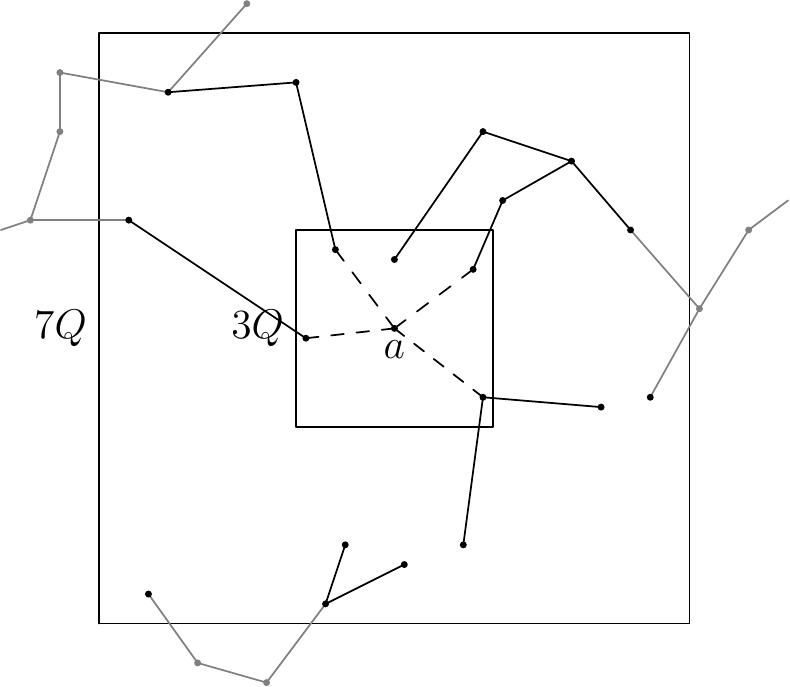}
\caption[Blah.]{The induction algorithm applied to a square $Q$.   The grey edges indicate the edges of $\Gamma_{\ell_0}(\ell)$ not included in the subgraph $\Gamma=(\Gamma_{\ell_0}(\ell))_Q$. The dashed lines indicate the edges added by applying the algorithm.  Note that in this case the graph $\Gamma$ has seven components, four of which intersect $3Q$.  The fixed point in $3Q\cap \mathcal{N}$ is denoted by $a$.}
 \end{figure}

 We carry out the inductive procedure  for $\ell=\ell_0, \dots, \tfrac{\ell(P)}{2}$, and thereby obtain the graph $\Gamma_{\ell_0}(\ell(P))$.  To continue our analysis, first note the following elementary fact:

\begin{lem}\label{dyadicfact}  Let $Q\in \mathcal{D}$ with $\ell(Q)=2\ell$.  For any two points $z_1, z_2\in 4Q$ with $|z_1-z_2|<\ell$, there is a dyadic square $Q'$ of sidelength $\ell$, such that  $7Q'\subset 7Q$, and  $z_1,z_2\in 3Q'$.\end{lem}

\begin{proof} Pick the square $Q'$ to  be the dyadic square of side length $\ell$ containing $z_1$.  Then $\text{dist}(Q', \mathbb{C}\backslash 3Q') =\ell$, so $z_2\in 3Q'$.

Since $\ell(Q)=2\ell$, we have that $4Q$ is a union of dyadic squares of side-length $\ell$.  Therefore $Q'$ is contained in $4Q$.  As the square annulus $7Q\backslash 4Q$ is of width $\tfrac{3}{2}\ell(Q)=3\ell$, we conclude that $7Q'\subset 7Q$.\end{proof}

We shall use this lemma (or rather a weaker statement with $4Q$ replaced by $3Q$) to deduce the following statement.

\begin{cla}\label{compsep}  For each $\ell\geq \ell_0$, and $Q\in \mathcal{D}$ with $\ell(Q)=2\ell$, any two connected components of $(\Gamma_{\ell_0}(\ell))_Q$ which intersect $3Q$ are $\ell$-separated in $3Q$.
\end{cla}

\begin{proof}First suppose $\ell=\ell_0$.  Let $z_1,z_2\in 3Q$ with $|z_1-z_2|< \ell_0$. Choose $Q'$ as in Lemma \ref{dyadicfact}.  We have that $z_1,z_2\in 3Q'$, and $3Q'\subset 7Q$.  But then $z_1$ and $z_2$ must have been joined when the base step rule was applied to the square $Q'$, and so they lie in the same component of $(\Gamma_{\ell_0}(\ell))_Q$.

Now suppose that $\ell>\ell_0$, and $z_1,z _2\in3Q$ with $|z_1-z_2|<\ell$.  Again, let $Q'$ be the square of Lemma \ref{dyadicfact}.  The induction step applied at level $\tfrac{\ell}{2}$ to the square $Q'$ ensures that $z_1$ and $z_2$ are joined by edges in $\mathcal{E}_{\ell_0}(\ell)$ that are contained in $7Q'\subset 7Q$.  Therefore, $z_1$ and $z_2$ lie in the same component of $(\Gamma_{\ell_0}(\ell))_Q$.\end{proof}


\begin{cla}\label{countcla}  There exists a constant $C>0$, such that for each $\ell\geq \ell_0$, and for every $Q\in \mathcal{D}$ with $\ell(Q)=2\ell$, the iterative procedure applied to $Q$ increases the length of $\Gamma_{\ell_0}(2\ell)$ by at most $C\ell$.\end{cla}

\begin{proof} From Claim \ref{compsep}, we see that the graph $(\Gamma_{\ell_0}(\ell))_Q$ can have at most $C$ components which have non-empty intersection with $3Q$.  
Consequently, the application of the inductive procedure can generate at most $C$ new edges, each of which having length no greater than $\sqrt{2}\cdot 6\ell$.  The claim follows.
\end{proof}

 \textbf{\emph{Adapting the graph to $P$.}}  
We begin with another observation about the induction algorithm.  Note that any two vertices in $3P\cap\mathcal{N}$ can be joined by edges in $\mathcal{E}_{\ell_0}(\ell(P))$
that are contained in $7P$.   Thus, the graph $(\Gamma_{\ell_0}(\ell(P)))_P$ has only one connected component which intersects $3P$, and we denote this component by $\Gamma$.  Let us denote by $L = L(\ell_0)$ the total length of $\Gamma$.

By Euler's theorem, there is a walk through $\Gamma$ which visits each vertex of $\Gamma$ at least once, and travels along each edge at most twice.  By a suitable parametrization of this walk, we arrive at the following lemma:

\begin{lem}\label{walklem}  There exists $F:[0,1]\rightarrow \mathbb{C}$, with $\|F\|_{\Lip}\leq 2L$, and such that $F([0,1])\supset \mathcal{N}\cap 3P$.
\end{lem}

If we have a suitable control over $L(\ell_0)$ independently of $\ell_0$, then $E\cap P$ is contained in the image of a Lipschitz graph:

 \begin{lem}\label{lipfun}  Suppose that there exists $M>0$ such that $L(\ell_0)\leq M\ell(P)$ for every $\ell_0>0$.  Then there exists $F:[0,1]\rightarrow \mathbb{C}$, such that $\|F\|_{\Lip}\leq 2M\ell(P)$, and $E\cap P \subset F([0,1])$.
 \end{lem}

 \begin{proof}  

 Let $\ell_0 = 2^{-k}$.   Let $F_k$ denote the function of Lemma \ref{walklem}.  
 Then $F_k([0,1])$ is a $\tau 2^{-k}$-net of $E\cap P$.

By appealing to the Arzela-Ascoli theorem, we see that there is a subsequence of the $F_k$ (which we again denote  by $F_k$), converging uniformly to some limit function $F$.  The function $F$ is Lipschitz continuous with Lipschitz norm no greater than $2M\ell(P)$.

 Now, for any $x\in E\cap P$, there exists a sequence $\{x_k\}_k$ where $x_k\in F_k([0,1])$, and $|x-x_k|<\tau 2^{-k}$.  Take $t_k\in [0,1]$ with $F_k(t_k)=x_k$.  There is a convergent subsequence of $\{t_k\}_k$ which converges to some $t\in [0,1]$.  But then $F(t)=x$, and the proof is complete.
 \end{proof}

We shall now estimate $L(\ell_0)$ in terms of the total side length of squares where the induction step has been carried out.  Note that only the base and inductive steps applied to the dyadic squares $Q$ contained in $7P$ can contribute to the length. 

 We shall first estimate the contribution to the length by the base step.

 \begin{cla}\label{basesteplength}  The contribution to the length of $\Gamma$ from the base step is no greater than $ C\tau^{-2}\ell(P)$.
\end{cla}

\begin{proof}Let $N$ denote the number of dyadic sub-squares of $7P$ with side length $\ell_0$ where the base step has been carried out.
 For any such square $Q$, we must have $3Q\cap \supp(\mu) \neq \varnothing$.  From the AD-regularity of $\mu$, it follows that $\mu(4Q)\geq c_0\ell(Q) = c_0\ell_0$.  Hence
 $$c_0\ell_0 N \leq \sum_{Q\in \mathcal{D}:\,Q\subset 7P , \,\ell(Q) = \ell_0}\mu(4Q) \leq C\mu(C P)\leq C\ell(P),$$
 and therefore $N\leq C\tfrac{\ell(P)}{\ell_0}$.    Consequently, the contribution to the length of $\Gamma$ from the base step is no greater than $C\tau^{-2}\ell_0 \tfrac{\ell(P)}{\ell_0}$, as required.\end{proof}

We now denote by $\mathcal{Q}(P,\ell_0)$ the collection of dyadic squares $Q\in \mathcal{D}$ such that $ \ell(Q)\in[\ell_0, \ell(P)]$, $Q\subset 7P$, and the inductive step has been carried out non-vacuously in $Q$ at scale $\tfrac{\ell(Q)}{2}$.

Claim \ref{countcla} guarantees that for each $Q\in \mathcal{Q}(P, \ell_0)$, an application of the inductive procedure increases the length $L(\ell_0)$ by no more than $C\ell(Q)$.  Combining this observation with Claim \ref{basesteplength}, we infer the following bound:
\begin{equation}\label{lbound}L(\ell_0) \leq C\tau^{-2}\ell(P) + C\sum_{Q\in \mathcal{Q}(P,\ell_0)} \ell(Q).
 \end{equation}

 \subsection{Bad squares}  Given the construction above, we would like to find a convenient way of identifying whether a square has been used in the inductive procedure at some scale.  Since we don't want these squares to occur very often, we call them \textit{bad squares}.

 \begin{defn}  We say that $Q\in \mathcal{D}$ is a $(\mu)$-bad square if there exist $\zeta, \xi \in B(z_Q, 10\ell(Q))\cap \supp(\mu)$, such that $|\zeta-\xi|\geq \ell(Q)/2$, and there exists $z\in[\zeta,\xi]$ such that $B(z, \tau\ell(Q))\cap E=\varnothing$.
 \end{defn}

 We now justify the use of this definition:

 \begin{lem}  Suppose that $\tau< \tfrac{1}{16}$.  
 Suppose that the inductive algorithm has been applied to $Q\in \mathcal{D}$.  Then $Q$ is a bad square.
 \end{lem}

\begin{proof}  

If the inductive algorithm has been applied, then there is a graph $\Gamma=(\mathcal{N},\mathcal{E})$\footnote{In the notation of the previous section, $\Gamma = \Gamma_{\ell_0}\bigl(\tfrac{\ell(Q)}{2}\bigl)$, for some $\ell_0\leq \tfrac{\ell(Q)}{2}$.}, with the following properties:
\begin{enumerate}
\item The set $\mathcal{N}$ forms a $\tfrac{\tau\ell(Q)}{2}$ net of $E$.
\item For every dyadic square $Q'$ with $\ell(Q')<\ell(Q)$ and $7Q'\subset 7Q$, we have that if $z_1,z_2\in 3Q'\cap \mathcal{N}$, then $z_1$ and $z_2$ lie in the same component of $\Gamma_Q$.
\item The connected components of $\Gamma_Q$ that intersect $3Q$ are at least $\tfrac{\ell(Q)}{2}$ separated in $3Q$.
\end{enumerate}
(In fact, property (2) implies property (3), as was seen in Claim \ref{compsep}).

By assumption, there exist two points $\zeta$ and $\xi$ in $3Q\cap \mathcal{N}$ that lie in different components of $\Gamma_Q$. Then $|\zeta-\xi|\geq\tfrac{\ell(Q)}{2}$. Consider the line segment $[\zeta,\xi]$.    Cover this segment with overlapping discs of radius $\tau\ell(Q)$, such that the centre of each disc lies in the line segment $[\zeta, \xi]$  (see Figure 4).

\begin{figure}[t]\label{BadSquarePic}
\centering
 \includegraphics[width=120mm]{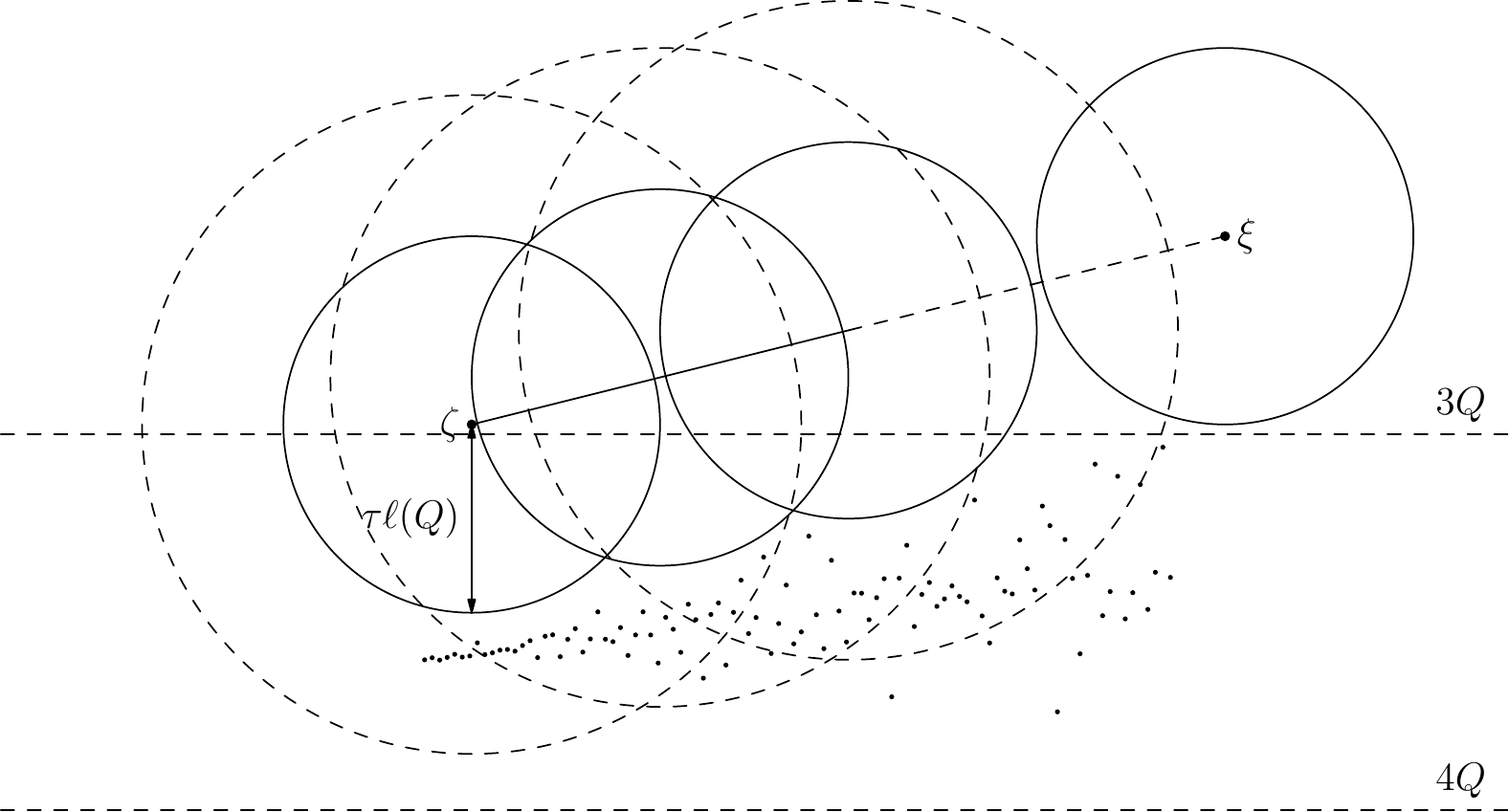}
\caption[The discs.]{The intersecting discs of radius $\tau\ell(Q)$, along with their concentric doubles.  The cloud of points represents those points of $\mathcal{N}$.}
 \end{figure}

 Suppose that every disc has positive $\mu$ measure.  If $\tau< \tfrac{1}{16}$, the concentric double of each disc is contained in $4Q$.  Furthermore, in the concentric double of each disc, there must be a point from $\mathcal{N}$.  We therefore form a chain of points in $\mathcal{N}\cap 4Q$, with every consecutive pair of points in the chain are separated by a distance of at most $8\tau\ell(Q)<\ell(Q)/2$.  Furthermore, the first point in the chain is within a distance of $\ell(Q)/2$ of  $\zeta$, and the last point in the chain is no further than $\ell(Q)/2$ from $\xi$.  Therefore,  Lemma \ref{dyadicfact} ensures that each consecutive pair of points in the chain are contained in the concentric triple of some dyadic square $Q'$ with $\ell(Q')<\ell(Q)$ and $7Q'\subset 7Q$.  But then property (2) yields that each such pair lies in the same component of $\Gamma_Q$.  As a result, $\zeta$ and $\xi$ also lie in the same component of $\Gamma_Q$.

 From this contradiction, we see that one of the discs of radius $\tau\ell(Q)$ has zero measure, which implies that $Q$ is a bad square.  \end{proof}

 Now let $\mathcal{B}^{\mu}$ denote the set of those squares $Q\in \mathcal{D}$ that are bad.  To prove Theorem \ref{thm1}, it suffices to prove the following proposition:

 \begin{prop} \label{carlreduction} Suppose that $\mu$ is a $C_0$-good measure with AD regularity constant $c_0$.  There is a constant $C=C(A,C_0,c_0)>0$ such that for each $P\in \mathcal{D}$,
 \begin{equation}\label{carl}\sum_{Q\in \mathcal{B}^{\mu}, \,Q\subset P} \ell(Q) \leq C\ell(P).
 \end{equation}
 \end{prop}

 Let us see how Theorem \ref{thm1} follows from this proposition.  Fix $P\in \mathcal{D}$, and construct $\Gamma_{\ell_0}(\ell(P))$ for $\ell_0<\ell(P)$.  From Proposition \ref{carlreduction}, the bound (\ref{lbound}) for the length $L(\ell_0)$ is no more than $M\ell(P)$, where $M$ can be chosen to depend on $A,c_0, C_0$, and $\tau$ (in particular, $M$ can be chosen independently of $P$).  But now Lemma \ref{lipfun} yields the existence of a function $F:[0,1]\rightarrow \mathbb{C}$ with Lipschitz norm no greater than $M\ell(P)$, such that $E\cap P\subset F([0,1])$.  This is the required uniform rectifiability.


 The condition (\ref{carl}) is very well known in harmonic analysis, and a family of squares $\mathcal{B}^{\mu}$ satisfying (\ref{carl}) is often referred to as a \textit{Carleson family}.  The best constant $C>0$ such that (\ref{carl}) holds for all $P\in \mathcal{D}$ is called the \textit{Carleson norm} of $\mathcal{B}^{\mu}$.

 \section{Bad squares and the Riesz family $\{\Psi^{\mu}_{Q,A}\}_Q$}

Fix a $C_0$-good measure $\mu$ with AD-regularity constant $c_0>0$.

Choose $A'>1$, with $A'\geq A$. Recall the Riesz family $\Psi^{\mu}_{Q,A}$ introduced in Section 5.  For each $Q\in \mathcal{D}$, we define
 $$\Theta_{A,A'}(Q)=\Theta_{A,A'}^{\mu}(Q) = \inf_{F\supset B(z_Q, A'\ell(Q))} \sup_{\psi\in\Psi^{\mu}_{Q,A}} \ell(Q)^{-1/2}|\langle \Cau_{\mu}(\chi_F), \psi\rangle_{\mu}|.
 $$
Consider a fixed $P\in \mathcal{D}$.  Then for each $Q\subset P$ there exists a function $\psi_Q\in\Psi^{\mu}_{Q,A}$ such that
$\Theta_{A,A'}(Q)^2\ell(Q) \leq 2 |\langle \Cau_{\mu}(\chi_{B(z_P,2 A'\ell(P))}), \psi_Q\rangle_{\mu}|^2 $ (note here that $B(z_Q, A'\ell(Q))\subset B(z_P, 2A'\ell(P))$ whenever $Q\subset P$).  Hence
$$\sum_{Q\subset P} \Theta_{A,A'}(Q)^2\ell(Q) \leq 2\sum_{Q\subset P} |\langle \Cau_{\mu}(\chi_{B(z_P, 2A'\ell(P))}), \psi_Q\rangle_{\mu}|^2.
$$
Since $\psi_Q$ ($Q\in \mathcal{D}$) forms a $C(C_0,A)$-Riesz system, the right hand side of this inequality is bounded by $C(C_0,A)\|\Cau_{\mu}(\chi_{B(z_P, 2A'\ell(P))})\|_{L^2(\mu)}^2$.  As $\mu$ is $C_0$-good, this quantity is in turn bounded by $C(C_0,A)\mu(B(z_P, 2A'\ell(P)))$, which is at most $ C(C_0,A,A')\ell(P)$.  Therefore
$$\sum_{Q\subset P}\Theta_{A,A'}(Q)^2\ell(Q)\leq C(C_0, A,A')\ell(P).
$$
As an immediate corollary of this discussion, we arrive at the following result:
\begin{lem}  Let $\gamma>0$.  Consider the set $\mathcal{F}_{\gamma}$ of dyadic squares $Q$ satisfying $\Theta_{A,A'}(Q)>\gamma$.  Then $\mathcal{F}_{\gamma}$ is a Carleson family, with Carleson norm bounded by $C(C_0,A,A')\gamma^{-2}$.
\end{lem}

In order to prove Proposition \ref{carlreduction} (from which Theorem \ref{thm1} follows), it therefore suffices to prove the following proposition:

\begin{prop}\label{badlower}  Suppose $\mu$ is a $C_0$-good measure with AD regularity constant $c_0>0$.  There exist constants $A,A'>1$, and $\gamma>0$, such that for any square $Q\in \mathcal{B}^{\mu}$,
$$\Theta^{\mu}_{A,A'}(Q)\geq \gamma.
$$
\end{prop}

We end this section with a simple remark about scaling.

\begin{rem}[Scaling Remark]\label{scalerem}
Fix a square $Q$, a function $\psi\in \Psi_{Q,A}^{\mu}$, and a compact set $F\subset \mathbb{C}$.  For $z_0\in \mathbb{C}$, set $\widetilde{\mu}(\cdot) = \tfrac{1}{\ell(Q)} \mu(\ell(Q)\cdot +z_0)$, $\widetilde\psi(\cdot) = \ell(Q)^{1/2}\psi(\ell(Q)\cdot + z_0)$ and $\widetilde{F} = \tfrac{1}{\ell(Q)}(F-z_0)$.  Then $||\widetilde{\psi}||_{\Lip}\leq 1$, $\supp(\widetilde\psi)\subset B(\tfrac{z_{Q}-z_0}{\ell(Q)}, A)$, and
$$\langle \Cau_{\widetilde{\mu}}(\chi_{\widetilde{F}}), \widetilde\psi\rangle_{\widetilde\mu} = \ell(Q)^{-1/2}\langle \Cau_{\mu}(\chi_F), \psi\rangle_{\mu}.$$
\end{rem}

\section{Reflectionless measures}\label{reflmeasintro}

In this section, we explore what happens if Proposition \ref{badlower} fails.   To do this, we shall need a workable definition of the Cauchy transform operator of a good measure acting on the constant function $1$.    Suppose that $\nu$ is a $C_0$-good measure with $0\not\in \supp(\nu)$.

\subsection{The function $\widetilde{\Cau}_{\nu}(1)$}  Let us begin with an elementary lemma.

\begin{lem}\label{l1awayfromsupp}  Suppose that $\sigma$ is a $C_0$-nice measure with $0\not\in\supp(\sigma)$. Let $z\in \mathbb{C}$ with $z\not\in\text{supp}(\sigma)$.  Set $d_0 = \text{dist}(\{0,z\}, \supp(\sigma))$.  Then
$$\int_{\mathbb{C}}\Bigl|\frac{1}{z-\xi}+\frac{1}{\xi}\Bigl| d\sigma(\xi) \leq \frac{C(C_0)|z|}{d_0}.
$$
\end{lem}

\begin{proof} Note the estimate
$$\int_{\mathbb{C}} \Bigl|\frac{1}{z-\xi} + \frac{1}{\xi}\Bigl|d\sigma(\xi) \leq \frac{2}{d_0} \sigma(B(0, 2|z|)) +2\int_{\mathbb{C}\backslash B(0, 2|z|)} \frac{|z|}{|\xi|^2}d\sigma(\xi).
$$
The first term on the right hand side has size no greater than $\tfrac{2C_0|z|}{d_0}$.  Since the domain of integration in the second term can be replaced by $\mathbb{C}\backslash B(0, \max(d_0,2|z|))$, Lemma \ref{tailest} guarantees that the second integral is bounded by $\tfrac{C|z|}{\max(2|z|,d_0)}$.
\end{proof}

For $z\not\in \supp(\nu)$, define
\begin{equation}\label{C1offsupp}\widetilde{\Cau}_{\nu}(1)(z) = \int_{\mathbb{C}}\Bigl[\frac{1}{z-\xi} + \frac{1}{\xi}\Bigl]d\mu(\xi) = \int_{\mathbb{C}}[K(z-\xi)-K(-\xi)]d\nu(\xi).
\end{equation}
Lemma \ref{l1awayfromsupp} guarantees that this integral converges absolutely.


 To extend the definition to the support of $\nu$, we shall follow a rather standard path.  We shall initially define $\widetilde{\Cau}_{\nu}(1)$ as a distribution, before showing it is a well defined function $\nu$-almost everywhere.  Recall from Section \ref{weaklims} how we interpret $\Cau_{\nu}$ as a bounded operator in $L^2(\nu)$.

 Fix $\psi\in \Lip_0(\mathbb{C})$.  Choose $\varphi\in \Lip_0(\mathbb{C})$ satisfying $\varphi\equiv 1$ on a neighbourhood of the support of $\psi$.  Then define
\begin{equation}\begin{split}\label{C1defn}\langle\widetilde{\Cau}_{\nu}(1),&\psi\rangle_{\nu} = \langle\Cau_{\nu}(\varphi),\psi\rangle_{\nu}- \Cau_{\nu}(\varphi)(0)\cdot\!\int_{\mathbb{C}}\psi d\nu\\
&+\int_{\mathbb{C}}\psi(z)\!\!\int_{\mathbb{C}}(1-\varphi(\xi))\bigl[K(z-\xi)\!-\!K(-\xi)\bigl] d\nu(\xi)d\nu(z).
\end{split}\end{equation}




Note that Lemma \ref{l1awayfromsupp}, applied with $\sigma = |1-\varphi|\cdot\nu$, yields that
$$\sup_{z\in \supp(\psi)}\int_{\mathbb{C}}|(1-\varphi(\xi))|\cdot |K(z-\xi)-K(-\xi)|d\nu(\xi)<\infty.
$$
Therefore the inner product in (\ref{C1defn}) is well defined.  We now claim that this inner product is independent of the choice of $\varphi$. To see this, let $\varphi_1$ and $\varphi_2$ be two compactly supported Lipschitz continuous functions, that are both identically equal to $1$ on some neighbourhood of $\supp(\psi)$.  If $z\in \supp(\psi)$, then
\begin{equation}\begin{split}\nonumber\int_{\mathbb{C}}&(1-\varphi_1(\xi))\Bigl[\frac{1}{z-\xi}+\frac{1}{\xi}\Bigl]d\nu(\xi)-\int_{\mathbb{C}}(1-\varphi_2(\xi))\Bigl[\frac{1}{z-\xi}+\frac{1}{\xi}\Bigl]d\nu(\xi)\\
& = \int_{\mathbb{C}}(\varphi_2(\xi)-\varphi_1(\xi))\Bigl[\frac{1}{z-\xi}+\frac{1}{\xi}\Bigl]d\nu(\xi)\\
& =\int_{\mathbb{C}}(\varphi_2(\xi)-\varphi_1(\xi))K(z-\xi)d\nu(\xi)+\Cau_{\nu}(\varphi_1)(0) - \Cau_{\nu}(\varphi_2)(0).
\end{split}\end{equation}
(All integrals in this chain of equalities converge absolutely.)  Now consider
$$\int_{\mathbb{C}}\psi(z)\int_{\mathbb{C}}(\varphi_2(\xi)-\varphi_1(\xi))K(z-\xi)d\nu(\xi)d\nu(z).
$$
As a result of the anti-symmetry of $K$, this equals
$$\frac{1}{2}\int_{\mathbb{C}}\int_{\mathbb{C}}K(z-\xi)\bigl[\psi(z)[\varphi_2(\xi)-\varphi_1(\xi)]- \psi(\xi)[\varphi_2(z)-\varphi_1(z)]\bigl]d\nu(\xi)d\nu(z).
$$
However, as we saw in Section \ref{weaklims}, $K(z-\xi)(\psi(z)\varphi_j(\xi) - \psi(\xi)\varphi_j(z))\in L^1(\nu\times\nu)$ for each $j=1,2$.  Hence, by using the linearity of the integral, and applying Fubini's theorem, we see that the last line equals $I_{\nu}(\varphi_2,\psi)-I_{\nu}(\varphi_1,\psi)$. By definition, this is equal to  $\langle \Cau_{\nu}(\varphi_2),\psi\rangle_{\nu} - \langle \Cau_{\nu}(\varphi_1),\psi\rangle_{\nu}$.  The claim follows.

We have seen that $\widetilde{\Cau}_{\nu}(1)$ is well-defined as a distribution.  For any bounded open set $U\subset\mathbb{C}$, if we choose $\varphi$ to be identically equal to $1$ on a neighbourhood of $U$, then $\widetilde{\Cau}_{\nu}(1)\in L^2(U, \nu)$.  Since Lipschitz functions with compact support are dense in $L^2(U, \nu)$, we find that $\widetilde{\Cau}_{\nu}(1)$ is well-defined $\nu$-almost everywhere.

Finally, we note that the smoothness of the function $\varphi$ is not essential.  If $\psi\in \Lip_0(\mathbb{C})$, let $U$ be a bounded open set containing $\supp(\psi)$.  Then it is readily seen that $\langle\widetilde{\Cau}_{\nu}(1),\psi\rangle_{\nu}$ equals
\begin{equation}\begin{split}\nonumber
\langle\Cau_{\nu}(\chi_U),\psi\rangle_{\nu}- \Cau_{\nu}(\chi_U)(0)\cdot\langle 1,\psi\rangle_{\nu} +\Bigl\langle\int_{\mathbb{C}\backslash U}\!\! [K(\cdot-\xi)+K(\xi)]d\nu(\xi),\psi\Bigl\rangle_{\nu}.
\end{split}\end{equation}

\subsection{The weak continuity of  $\widetilde{\Cau}_{\nu}(1)$}

We shall introduce a couple more sets of functions.  $\Phi^{\nu}_{A}$  will denote those functions $\psi$ with $\|\psi\|_{\Lip}\leq 1$, that satisfy $\int_{\mathbb{C}} \psi \,d\nu =0$  and $\supp(\psi)\subset B(0,A)$.  We define $\Phi^{\nu}$ to be the set of compactly supported functions $\psi$ with $\|\psi\|_{\Lip}\leq 1$, satisfying $\int_{\mathbb{C}} \psi \,d\nu =0$.

We start with another standard estimate.

\begin{lem}\label{l1meanzero}
Suppose that $\nu$ is $C_0$-nice measure.  For $R>0$, suppose that $\psi\in \Phi_{R}^{\nu}$.  Then $\|\psi\|_{L^1(\nu)}\leq C(C_0)R^2$.
\end{lem}

\begin{proof}  Simply note that
$$\int_{B(0,R)}|\psi|d\nu = \int_{B(0,R)} \Bigl|\psi-\frac{1}{\nu(B(0,R))}\int_{B(0,R)} \psi d\nu\Bigl| d\nu.
$$
This quantity is no greater than $\text{osc}_{B(0,R)}(\psi) \nu(B(0,R))$, which is less than or equal to $2R\cdot C_0R$.
\end{proof}

Our next lemma concerns a weak continuity property of $\widetilde{\Cau}_{\nu}(1)$.

\begin{lem}\label{weakcontcau1}  Let $\nu_k$ be a sequence of $C_0$-good measures, with $0\not\in \supp(\nu_k)$.  Suppose that $\nu_k$ converge weakly to $\nu$ (and so $\nu$ is $C_0$-good), with $0\not\in \supp(\nu)$.  Fix  non-negative sequences $\widetilde{\gamma}_k$ and $\widetilde{A}_k$, satisfying $\widetilde{\gamma}_k\rightarrow 0$, and $\widetilde{A}_k\rightarrow \widetilde{A}\in (0,\infty]$.

If $|\langle \widetilde{\Cau}_{\nu_k}(1), \psi\rangle_{\nu_k}| \leq \widetilde{\gamma}_k$ for all $\psi\in \Phi^{\nu_k}_{\widetilde{A}_k}$, then $$|\langle \widetilde{\Cau}_{\nu}(1), \psi\rangle_{\nu}| =0 \text{ for all }\psi\in \Phi^{\nu}_{\widetilde{A}}.$$
(Here $\Phi^{\nu}_{\widetilde{A}}= \Phi^{\nu}$ if $\widetilde{A}=\infty$.)
\end{lem}

\begin{proof}  If $\nu(B(0,\widetilde{A}))=0$, then there is nothing to prove, so assume that $\nu(B(0,\widetilde{A}))>0$.  Let $\eps>0$.  Pick $\psi\in \Phi^{\nu}_{\widetilde{A}}$.   Then there exists $R\in (0,\infty)$ such that $\supp(\psi)\subset B(0,R)\subset B(0, \widetilde{A}_k)$ for all sufficiently large $k$, and $\nu(B(0,R))>0$.

Fix $\rho\in \Lip_0$ with $\supp(\rho)\subset B(0,R)$, such that $\int_{\mathbb{C}}\rho d\nu=c_{\rho}>0$.  Define
$$\psi_k = \psi- b_k \rho, \text{ with }b_k = \frac{1}{\int_{\mathbb{C}}\rho\, d\nu_k} \int_{\mathbb{C}}\psi d\nu_k.
$$
Note that $\psi_k$ is  supported in $B(0,R)$, and has $\mu_k$-mean zero.  Since $b_k\rightarrow 0$, we have that $\|\psi_k\|_{\Lip} \leq 2$ for all sufficiently large $k$.  Therefore, for these $k$, we have $|\langle \widetilde{\Cau}_{\nu_k}(1), \psi_k\rangle_{\nu_k}| \leq 2\widetilde{\gamma}_k$.

Now pick $\varphi\in \Lip_0$ with $\varphi \equiv 1$ on $B(0,2R)$ and $0\leq \varphi\leq 1$ on $\mathbb{C}$, such that both
$$|\langle \Cau_{\nu_k}(\varphi), \psi_k\rangle_{\nu_k} - \langle \widetilde{\Cau}_{\nu_k}(1), \psi_k\rangle_{\nu_k}|<\eps,$$ for all sufficiently large $k$, and $$|\langle \Cau_{\nu}(\varphi), \psi\rangle_{\nu} - \langle \widetilde{\Cau}_{\nu}(1), \psi\rangle_{\nu}|<\eps.
$$

To see that such a choice is possible, note that if $\varphi \equiv 1$ on $B(0,R')$ for $R'>2R$, then
$|\langle \Cau_{\nu_k}(\varphi), \psi_k\rangle_{\nu_k} - \langle \widetilde{\Cau}_{\nu_k}(1), \psi_k\rangle_{\nu_k}|$ is bounded by
$$\int_{B(0,R)}|\psi_k(z)|\int_{\mathbb{C}}|1-\varphi(\xi)|  \Bigl|\frac{1}{z-\xi} + \frac{1}{\xi}\Bigl|d\nu_k(\xi)d\nu_k(z),
$$
(recall here that $\psi_k$ has $\nu_k$ mean zero).  For any $z\in B(0,R)$, note that $\text{dist}(z, \supp(1-\varphi))\geq \tfrac{R'}{2}$, and so by applying Lemma \ref{l1awayfromsupp}, we see that the above quantity is no greater than $C\|\psi_k\|_{L^1(\nu_k)}\tfrac{R}{R'}$.  Applying Lemma \ref{l1meanzero}, we see that  $|\langle \Cau_{\nu_k}(\varphi), \psi_k\rangle_{\nu_k} - \langle \widetilde{\Cau}_{\nu_k}(1), \psi_k\rangle_{\nu_k}|\leq C\tfrac{R^3}{R'}$, which can be made smaller than $\eps$ with a reasonable choice of $R'$.  The same reasoning shows that $|\langle \Cau_{\nu}(\varphi), \psi\rangle_{\nu} - \langle \widetilde{\Cau}_{\nu}(1), \psi\rangle_{\nu}|<\eps$ provided $R'$ is chosen suitably.

On the other hand, as $\nu_k$ is $C_0$-good, we have $ |\langle \Cau_{\nu_k}(\varphi), \rho \rangle_{\nu_k}|\leq C_0\|\varphi\|_{L^2(\nu_k)}\|\rho\|_{L^2(\nu_k)} $.   Since $\varphi$ and $\rho$ are compactly supported Lipschitz functions, the right hand side of this inequality converges to $C_0\|\varphi\|_{L^2(\nu)}\|\rho\|_{L^2(\nu)} $, and so it is bounded independently of $k$.

Bringing together these observations, we see that $\langle \Cau_{\nu_k}(\varphi), \psi_k\rangle_{\nu_k}$ converges to $\langle \Cau_{\nu}(\varphi), \psi\rangle_{\nu}$ as $k\rightarrow \infty$.  But since $|\langle \widetilde{\Cau}_{\nu_k}(1), \psi_k\rangle_{\nu_k}| \leq 2\widetilde{\gamma}_k$ for $k$ large enough, we deduce from the triangle inequality that $|\langle \widetilde{\Cau}_{\nu}(1), \psi\rangle_{\nu}| \leq 4\eps$.
\end{proof}

Let us now suppose that Proposition \ref{badlower} is false.  Fix $A\geq 100$.  For each $k \in \mathbb{N}$, $k\geq 2A$, there is a $C_0$-good measure $\mu_k$ with AD-regularity constant $c_0>0$, a square $Q_k\in \mathcal{B}^{\mu_k}$, and a set $E_k\supset B(z_{Q_k}, k\ell(Q_k))$ such that
\begin{equation}\label{gammaksmall}| \langle \Cau_{\mu_k}(\chi_{E_k}), \psi \rangle_{\mu_k}|\leq \frac{1}{k}, \text{ for all }\psi\in \Psi_{A,Q_k}^{\mu_k}.
\end{equation}
In addition, by the scale invariance of the condition (\ref{gammaksmall}) (see Remark \ref{scalerem}), we may dilate and translate the square $Q_k$ so that it has side length $1$, and so that there are $\zeta_k, \xi_k \in B(z_{Q_k}, 10)\cap\supp(\mu_k)$ with $|\zeta_k-\xi_k|\geq 1/2$, such that $0\in[\zeta_k,\xi_k]$ and $B(0, \tau)\cap \supp(\mu_k)=\varnothing$.  Note that the translated and dilated square is not necessarily dyadic.


By passing to a subsequence if necessary, we may assume that $\mu_k$ converge weakly to a measure $\mu^{(A)}$ (using the uniform niceness of the $\mu_k$).  This limit measure is $C_0$-good, with AD-regularity constant $c_0$, and $0\not\in \mu^{(A)}$.  Furthermore, it is routine to check that $\mu^{(A)}$ satisfies the following property (recall Lemma \ref{ADsupp}):
\begin{equation}\label{badmuA}\begin{split}\text{There exist }&\xi, \, \zeta  \in \overline{B(0,20)}\cap \supp(\mu^{(A)}),\text{ with } |\xi-\zeta|\geq \frac{1}{2},\\&\text{  such that }0\in [\zeta,\xi]\text{ and }B(0,\tau)\cap \supp(\mu^{(A)})=\varnothing.\end{split}\end{equation}

Now, for each $k$ we have that $B(0, \tfrac{A}{2}) \subset B(z_{Q_k}, A)$ and $E_k\supset B(0, \tfrac{k}{2}) \supset B(0,A)$.  We claim that 
$$|\langle \widetilde{\Cau}_{\mu_k}(1), \psi\rangle_{\mu_k}| \leq \frac{1}{k} +\frac{CA^3}{k}, \text{ for all }\psi\in \Phi_{\tfrac{A}{2}}^{\mu_k}.
$$
To see this, note that for any $\psi\in \Phi_{\tfrac{A}{2}}^{\mu_k}$, $\langle \widetilde{\Cau}_{\mu_k}(1), \psi\rangle_{\mu_k}$ is equal to
$$ \langle \Cau_{\mu_k}(\chi_{E_k}), \psi\rangle_{\mu_k} + \int_{B(0,\tfrac{A}{2})}\psi(z)\int_{\mathbb{C}\backslash E_k} \Bigl(\frac{1}{z-\xi} + \frac{1}{\xi}\Bigl) d\mu_k(\xi)d\mu_k(z).
$$
The first term is smaller than $\tfrac{1}{k}$ in absolute value.  To bound the second term, note that for any $z\in B(0, \tfrac{A}{2})$, $\text{dist}(z, \mathbb{C}\backslash E_k)\geq \tfrac{k}{2}$ so Lemma \ref{l1awayfromsupp} yields that this second term is no larger than $\tfrac{CA}{k}\|\psi\|_{L^1(\mu_k)}$, and applying Lemma \ref{l1meanzero} yields the required estimate.

We now apply Lemma \ref{weakcontcau1} with $\nu_k = \mu_k$, $\widetilde{\gamma}_k =  \tfrac{1}{k} +\tfrac{CA^3}{k}$, and $\widetilde{A}_k = \tfrac{A}{2}$.  Our conclusion is that $|\langle \widetilde{\Cau}_{\mu^{(A)}}(1), \psi\rangle_{\mu^{(A)}}| = 0,
$ for all $\psi\in \Phi_{\tfrac{A}{2}}^{\mu^{(A)}}$.

We now set $A=k$, for $k>100$.  The above argument yields a measure $\mu^{(k)}$ satisfying $|\langle \widetilde{\Cau}_{\mu^{(k)}}(1), \psi\rangle_{\mu^{(k)}}| = 0,
$ for all $\psi\in \Phi_{\tfrac{k}{2}}^{\mu^{(k)}}$.   We now pass to a subsequence of $\{\mu^{(k)}\}_k$ so that $\mu^{(k)}\rightarrow \mu$ weakly as $k\rightarrow \infty$.  The measure $\mu$ is $C_0$-good with AD-regularity constant $c_0$, and satisfies the property (\ref{badmuA}) with $\mu$ replacing $\mu^{(A)}$.  By applying Lemma \ref{weakcontcau1}  with $\widetilde{\nu_k} = \mu^{(k)}$, $\widetilde{\nu}=\mu$, $\widetilde{A}_k = \tfrac{k}{2}$, and $\widetilde{\gamma}_k=0$, we arrive at the following result:
\begin{lem}\label{absurd} Suppose that Proposition \ref{badlower} fails.  Then there exists a $C_0$-good measure $\mu$ with AD-regularity constant $c_0$, such that
\begin{equation}\label{reflectionless} |\langle \widetilde{\Cau}_{\mu}(1), \psi\rangle_{\mu}| = 0, \text{ for all }\psi\in \Phi^{\mu},
\end{equation}
and there exist $\xi, \, \zeta \in \overline{B(0,20)}\cap \supp(\mu)$, with  $|\xi-\zeta|\geq \frac{1}{2},$ such that $0\in [\zeta,\xi]$ and $B(0,\tau)\cap \supp(\mu)=\varnothing.$

\end{lem}

We call any measure $\mu$ that satisfies (\ref{reflectionless}) a \textit{reflectionless measure}.   It turns out that there aren't too many good AD-regular reflectionless measures.

\begin{prop} \label{reflectioncharac} Suppose that $\mu$ is a non-trivial reflectionless good AD-regular measure.  Then $\mu = c\mathcal{H}^1_{L}$ for a line $L$, and a positive constant $c>0$.
\end{prop}

Note that Proposition \ref{reflectioncharac} contradicts the existence of the measure $\mu$ in Lemma \ref{absurd}.  Therefore, once Proposition \ref{reflectioncharac} is proved, we will have asserted Proposition \ref{badlower}, and Theorem \ref{thm1} will follow.  Hence it remains to prove the proposition.   It is at this stage where the precise structure of the Cauchy transform is used.

\section{The Cauchy transform of a reflectionless good measure $\mu$ is constant in each component of $\mathbb{C}\backslash \supp(\mu)$}

Our goal is now to prove Proposition \ref{reflectioncharac}.  Suppose that $\mu$ is a reflectionless $C_0$-good measure.  We may assume that $0\not\in \supp(\mu)$.  All constants in this section may depend on $C_0$ without explicit mention.


Since $\widetilde{\Cau}_{\mu}(1)$ is a well defined $\mu$-almost everywhere function and satisfies (\ref{reflectionless}), we conclude that it is a constant $\mu$-almost everywhere in $\mathbb{C}$, say with value $\kap\in \mathbb{C}$.  

\begin{lem} \label{refconstsupp} Suppose that $\mu$ is a $C_0$-good reflectionless measure, and $0\not\in \supp(\mu)$.  Then there exists $\kap\in \mathbb{C}$ such that
 $\widetilde{\Cau}_{\mu}(1)=\kap$  $\mu$-almost everywhere.
\end{lem}


Our considerations up to now have been quite general, but now our hand is forced to use the magic of the complex plane.  The main difficulty is to obtain some information about the values of $\widetilde{\Cau}_{\mu}(1)$ away from the support of $\mu$ in terms of the constant value $\kap$.

\subsection{The resolvent identity}



\begin{lem}\label{resolvlem} For every $z\not\in\supp(\mu)$,
$$[\widetilde{\Cau}_{\mu}(1)(z)]^2 = 2\kap\cdot   \widetilde{\Cau}_{\mu}(1)(z).
$$
\end{lem}

An immediate consequence of Lemma \ref{resolvlem} is that either $\widetilde{\Cau}_{\mu}(1)(z)=2\kap$ or $\widetilde{\Cau}_{\mu}(1)(z)=0$ for any $z\not\in \supp(\mu)$.  Since $\widetilde{\Cau}_{\mu}(1)$ is a continuous function away from $\supp(\mu)$, it follows that $\widetilde{\Cau}_{\mu}(1)$ is constant in each connected component of $\mathbb{C}\backslash \supp(\mu)$.

A variant of Lemma \ref{resolvlem}, where the Cauchy transform is considered in the sense of principal value, has previously appeared in work of Melnikov, Poltoratski, and Volberg, see Theorem 2.2 of \cite{MPV}.  We shall modify the proof from \cite{MPV} in order to prove Lemma \ref{resolvlem}.

We shall first provide an incorrect proof of this lemma.  Indeed, note the following regularized version of the resolvent identity: for any three distinct points $z,\xi, \omega\in \mathbb{C}$,
\begin{equation}\label{resolve}\begin{split}\Bigl[\frac{1}{z-\xi}&+\frac{1}{\xi}\Bigl]\cdot\Bigl[\frac{1}{\xi-\omega}+\frac{1}{\omega}\Bigl] +\Bigl[ \frac{1}{z-\omega}+\frac{1}{\omega}\Bigl]\cdot\Bigl[\frac{1}{\omega-\xi}+\frac{1}{\xi}\Bigl] \\
&=\Bigl[ \frac{1}{z-\xi}+\frac{1}{\xi}\Bigl]\cdot\Bigl[ \frac{1}{z-\omega}+\frac{1}{\omega}\Bigl].
\end{split}\end{equation}
Integrating both sides of this equality with respect to $d\mu(\xi)d\mu(\omega)$, we (only formally!) arrive at $2\widetilde{\Cau}_{\mu}(\widetilde{\Cau}_{\mu}(1))(z) = [\widetilde{\Cau}_{\mu}(z)]^2$.  Once this is established, Lemma \ref{refconstsupp} completes the proof.  The proof that follows is a careful justification of this integration.

\begin{proof}We shall define $\widetilde{\Cau}_{\mu, \delta}(\varphi)(\omega) = \int_{\mathbb{C}}\bigl[K_{\delta}(\omega - \xi) +\tfrac{1}{\xi}\bigl] \varphi(\xi)d\mu(\xi)$ for $\varphi\in \Lip_0(\mathbb{C})$.  In particular, as $\delta$ tends to $0$, $\widetilde{\Cau}_{\mu, \delta}(\varphi)$ converges to $\widetilde{\Cau}_{\mu}(\varphi)=\Cau_{\mu}(\varphi) - \Cau_{\mu}(\varphi)(0)$ weakly in $L^2(\mu)$.

Set $d_0 = \text{dist}(\{z,0\}, \supp(\mu))$.  Snce $d_0>0$, Lemma \ref{l1awayfromsupp} tells us that $[K(z-\cdot)+K(\cdot)]\in L^1(\mu)$.


For $N>0$, define a bump function $\varphi_N\in \Lip_0(\mathbb{C})$, satisfying $\varphi_N\equiv 1$ on $B(0, N)$, and $\supp(\varphi)\subset B(0, 2N)$.  Consider the identity (\ref{resolve}), and multiply both sides by $\varphi_N(\xi)\varphi_N(\omega)$.  After integration against $d\mu(\xi)d\mu(\omega)$, the right hand side of this equality becomes
$$\int_{\mathbb{C}}\Bigl[\frac{1}{z-\xi} + \frac{1}{\xi}\Bigl]\varphi_N(\xi)d\mu(\xi)\int_{\mathbb{C}}\Bigl[\frac{1}{z-\omega} + \frac{1}{\omega}\Bigl]\varphi_N(\omega)d\mu(\omega).
$$
But since $\displaystyle\bigl[K(z-\cdot)+K(\cdot)\bigl]\in L^1(\mu)$, the dominated convergence theorem ensures that as $N\rightarrow \infty$, this expression converges to $[\widetilde{\Cau}_{\mu}(1)(z)]^2$.

Now, let $\delta>0$, and note that
$$\frac{1}{\xi-\omega} = K_{\delta}(\xi-\omega)+ \chi_{B(0,\delta)}(\xi-\omega)\cdot\Bigl[\frac{1}{\xi-\omega}-\frac{\overline{\xi-\omega}}{\delta^2}\Bigl].
$$
Consider the integral
\begin{equation}\begin{split}\nonumber\int_{\mathbb{C}}\int_{\mathbb{C}} \chi_{B(0,\delta)}(\xi-\omega)&\Bigl[\frac{1}{\xi-\omega}-\frac{\overline{\xi-\omega}}{\delta^2}\Bigl]\cdot\Bigl[\frac{1}{z-\xi}+\frac{1}{\xi}-\frac{1}{z-\omega} - \frac{1}{\omega}\Bigl]\\
& \varphi_N(\xi)\varphi_N(\omega)d\mu(\xi)d\mu(\omega).
\end{split}\end{equation}
Note that $\bigl|\frac{1}{z-\xi}+\frac{1}{\xi}-\frac{1}{z-\omega} - \frac{1}{\omega}\bigl|\leq \frac{2}{d_0^2}|\xi-\omega|$ for $\xi, \omega\in \supp(\mu)$, and so this integral is bounded in absolute value by a constant multiple of $\int_{\mathbb{C}}\varphi_N(\xi) \mu(B(\xi, \delta))d\mu(\xi)$, which is bounded by $C\delta N.$  This converges to zero as $\delta\rightarrow 0$.

Making reference to (\ref{resolve}), we have thus far shown that
\begin{equation}\begin{split}\nonumber\lim_{N\rightarrow\infty}\lim_{\delta\rightarrow 0}&\int_{\mathbb{C}}\int_{\mathbb{C}}\varphi_N(\xi)\varphi_N(\omega)\Bigl\{\Bigl[\frac{1}{z-\xi}+\frac{1}{\xi}\Bigl]\Bigl[K_{\delta}(\xi-\omega)+\frac{1}{\omega}\Bigl] \\
&+\Bigl[ \frac{1}{z-\omega}+\frac{1}{\omega}\Bigl]\Bigl[K_{\delta}(\omega-\xi)+\frac{1}{\xi}\Bigl]\Bigl\}d\mu_N(\xi)d\mu_N(\omega)= [\widetilde{\Cau}_{\mu}(1)(z)]^2.
\end{split}\end{equation}

By Fubini's theorem, and the weak convergence of $\widetilde{\Cau}_{\mu, \delta}(\varphi)$ to $\widetilde{\Cau}_{\mu}(\varphi)$, the left hand side of this equality is equal to twice the following limit
\begin{equation}\begin{split}\lim_{N\rightarrow \infty}\Bigl[ \int_{\mathbb{C}} \varphi_N(\xi)\Bigl[\frac{1}{z-\xi} + \frac{1}{\xi}\Bigl]\widetilde{\Cau}_{\mu}(\varphi_N)(\xi) d\mu(\xi)\Bigl].
\end{split}\end{equation}
Therefore, to prove the lemma, it suffices to show that this limit equals $\kap \widetilde{\Cau}_{\mu}(1)(z)$.

To do this, let $\alpha\in (\tfrac{1}{2},1)$.  First consider
$$I_N = \int_{B(0, N^{\alpha})} \varphi_N(\xi)\Bigl|\frac{1}{z-\xi} + \frac{1}{\xi}\Bigl|\cdot|\widetilde{\Cau}_{\mu}(\varphi_N)(\xi) -\widetilde{\Cau}_{\mu}(1)(\xi)|d\mu(\xi).
$$
Note that, for $|\xi|\leq N^{\alpha}$, we have
$$|\widetilde{\Cau}_{\mu}(\varphi_N)(\xi) -\widetilde{\Cau}_{\mu}(1)(\xi)|\leq \int_{\mathbb{C}}|1-\varphi_N(\omega)|\Bigl|\frac{1}{\xi-\omega}-\frac{1}{\omega}\Bigl|d\mu(\omega).
$$
Applying Lemma \ref{l1awayfromsupp} yields an upper bound for the right hand side of $C\tfrac{|\xi|}{N}\leq CN^{\alpha-1}.$  But as $[K(z-\cdot)+K(\cdot)]\in L^1(\mu)$, we conclude that $I_N\rightarrow 0$ as $N\rightarrow \infty$.  Next, note that
$$\int_{B(0, N^{\alpha})} \!\!\varphi_N(\xi)\Bigl[\frac{1}{z-\xi} + \frac{1}{\xi}\Bigl]\widetilde{\Cau}_{\mu}(1)(\xi) d\mu(\xi)\! = \!\kap\! \int_{B(0, N^{\alpha})} \!\!\varphi_N(\xi)\Bigl[\frac{1}{z-\xi} + \frac{1}{\xi}\Bigl] d\mu(\xi),
$$
which converges to $\kap\cdot  \widetilde{\Cau}_{\mu}(1)(z)$ as $N\rightarrow \infty$.

To complete the proof the lemma, it now remains to show that
$$\lim_{N\rightarrow \infty}\int_{B(0,2N)\backslash B(0, N^{\alpha})}|\widetilde{\Cau}_{\mu}(\varphi_N)(\xi)|\cdot \Bigl|\frac{1}{z-\xi}-\frac{1}{\xi}\Bigl|d\mu(\xi) =0.
$$
To do this, first note that $|\widetilde{\Cau}_{\mu}(\varphi_N)(\xi)|\leq |\Cau_{\mu}(\varphi_N)(\xi)| + C\log \tfrac{N}{d_0}$ (this merely uses the $C_0$-niceness of $\mu$).  On the other hand, for sufficiently large $N$, $\bigl|\frac{1}{z-\xi}-\frac{1}{\xi}\bigl|\leq \frac{8|z|}{N^{2\alpha}}$ for $|\xi|\geq N^{\alpha}$.  Therefore, there is a constant $C=C(C_0,d_0)>0$ such that
\begin{equation}\begin{split}\nonumber&\int_{B(0,2N)\backslash B(0, N^{\alpha})}|\widetilde{\Cau}_{\mu}(1)(\xi)|\cdot \Bigl|\frac{1}{z-\xi}-\frac{1}{\xi}\Bigl|d\mu(\xi)\\
&\;\;\;\leq \frac{C|z|\log N}{N^{2\alpha}}\mu(B(0,2N))+ \frac{C|z|}{N^{2\alpha}}\int_{ B(0,2N)}|\Cau_{\mu}(\varphi_N)(\xi)|d\mu(\xi),
\end{split}\end{equation}
Finally, since  $\|\Cau_{\mu}(\varphi_N)\|_{L^2(\mu)} \leq C\sqrt{\mu(B(0, 2N))}$, and $\mu(B(0,2N))\leq CN$, we estimate the right hand side here by a constant multiple of $\frac{|z|N\log N}{N^{2\alpha}}$, which tends to zero as $N\rightarrow \infty$.\end{proof}

\section{The proof of Proposition \ref{reflectioncharac}}

In this section we conclude our analysis by proving Proposition \ref{reflectioncharac}.  To do this, we shall use the notion of a tangent measure, which was developed by Preiss \cite{P87}.  Suppose that $\nu$ is a Borel measure on $\Comp$.  The measure $\nu_{z,\lambda}(A) = \tfrac{\nu(\lambda A+z)}{\lambda}$ is called a $\lambda$-blowup of  $\nu$ at $z$.  A tangent measure of $\nu$ at $z$ is any measure that can be obtained as a weak limit of a sequence of $\lambda$-blowups of $\nu$ at $z$ with $\lambda \rightarrow 0$.

Now suppose that $\mu$ is a nontrivial $C_0$-good with AD regularity constant $c_0$.  Then any $\lambda$-blowup measure of $\mu$ at $z\in \supp(\mu)$ will again have these properties ($C_0$-goodness, and $c_0$-AD regularity).  Therefore, both properties are inherited by any tangent measure of $\mu$.  In particular, every tangent measure of $\mu$ at $z\in \supp(\mu)$ is non-trivial, provided that $\mu$ is non-trivial.  
Lastly, we remark that if $\mu$ is reflectionless, then any tangent measure of $\mu$ is also reflectionless.  This follows from a simple application of Lemma \ref{weakcontcau1}.

In what follows, it will often be notationally convenient to translate a point on $\supp(\mu)$ to the origin. Whenever this is the case, the definition of $\widetilde{\Cau}_{\mu}(1)$ in (\ref{C1offsupp}) is translated with the support of the measure, and becomes \begin{equation}\label{z0defn}\widetilde{\Cau}_{\mu}(1)(z) = \int_{\mathbb{C}}[K(z-\xi)-K(z_0-\xi)]d\mu(\xi),\end{equation}
for some $z_0\not\in\supp(\mu)$. If $\mu$ is reflectionless, then $\widetilde{\Cau}_{\mu}(1)$ is constant in each component of $\mathbb{C}\backslash \supp(\mu)$, and takes one of two values in $\mathbb{C}\backslash \supp(\mu)$.

\subsection{Step 1}  Suppose that $\supp(\mu)\subset L$, for some line $L$.  Then by translation and rotation we may as well assume that $L=\mathbb{R}$.  If the support is not the whole line, then there exists an interval $(x, x')$ disjoint from the support of $\mu$, with either $x$ or $x'$ in the support of $\mu$.  By rotating the support if necessary, we may assume that  $x'\in \supp(\mu)$.

Denote by $\tilde\mu$ a non-zero tangent measure of $\mu$ at $x'$.  
Then $\tilde\mu$ has support contained in the segment $[x',\infty)$, and $x'\in \supp(\tilde\mu)$.  Since $\tilde\mu$ is reflectionless, we may apply Lemma \ref{resolvlem} to deduce that $\widetilde{\Cau}_{\tilde\mu}(1)(x'-t)$ is constant for all $t>0$.  Differentiating this function with respect to $t$, we arrive at $\int_{x'}^{\infty}\frac{1}{(x-t-y)^2}d\tilde\mu(y).$
This integral is strictly positive as $\tilde\mu$ is not identically zero.  From this contradiction we see that $\supp(\mu)=\mathbb{R}$.

Consequently, we have that $d\mu(t) = h(t) dt$, where $c_0\leq h(t)\leq C_0$.  Now let $y>0$ and consider, for $x\in \mathbb{R}$,
$$\widetilde{\Cau}_{\mu}(1)(x-yi) - \widetilde{\Cau}_{\mu}(1)(x+yi) = 2i\int_{\mathbb{R}}\frac{y}{(x-t)^2+y^2}h(t) dt.
$$
The expression on the left hand side is constant in $x\in \mathbb{R}$ and $y>0$.  On the other hand, the integral on the right hand side is a constant multiple of the harmonic extension of $h$ to $\mathbb{R}^2_+$.  The Poisson kernel is an approximate identity, and so by letting $y\rightarrow0^{+}$ we conclude that $h$ is a constant.  Therefore $\mu = c m_1,$ with $c>0$.

\subsection{Step 2}  We now turn to the general case.  We first introduce some notation.  For $z\in \mathbb{C}$ and a unit vector $e$, $H_{z, e}$ denotes the (closed) half space containing $z$ on the boundary, with inner unit normal $e$.   With $\alpha\in (0,1)$, we denote $C_{z,e}(\alpha) = \{\xi\in \mathbb{C}: \langle \xi-z , e\rangle > \alpha |\xi-z|\}$, where $\langle \cdot, \cdot \rangle$ is the standard inner product in $\mathbb{R}^2$.

\begin{lem}\label{nothingotherside}  Suppose that $z\not\in \supp(\mu)$.  Let $\tilde{z}$ be a closest point in $\supp(\mu)$ to $z$, and set $e = \tfrac{\tilde{z}-z}{|\tilde{z}-z|}$.  For each $\alpha\in (0,1)$, there is a radius $r_{\alpha}>0$ such that $B(\tilde z, r_{\alpha})\cap C_{\tilde z, e}(\alpha)$ is disjoint from $\supp(\mu)$.
\end{lem}

\begin{proof}

\begin{figure}[t]\label{ballpic}
\centering
 \includegraphics[width = 108mm]{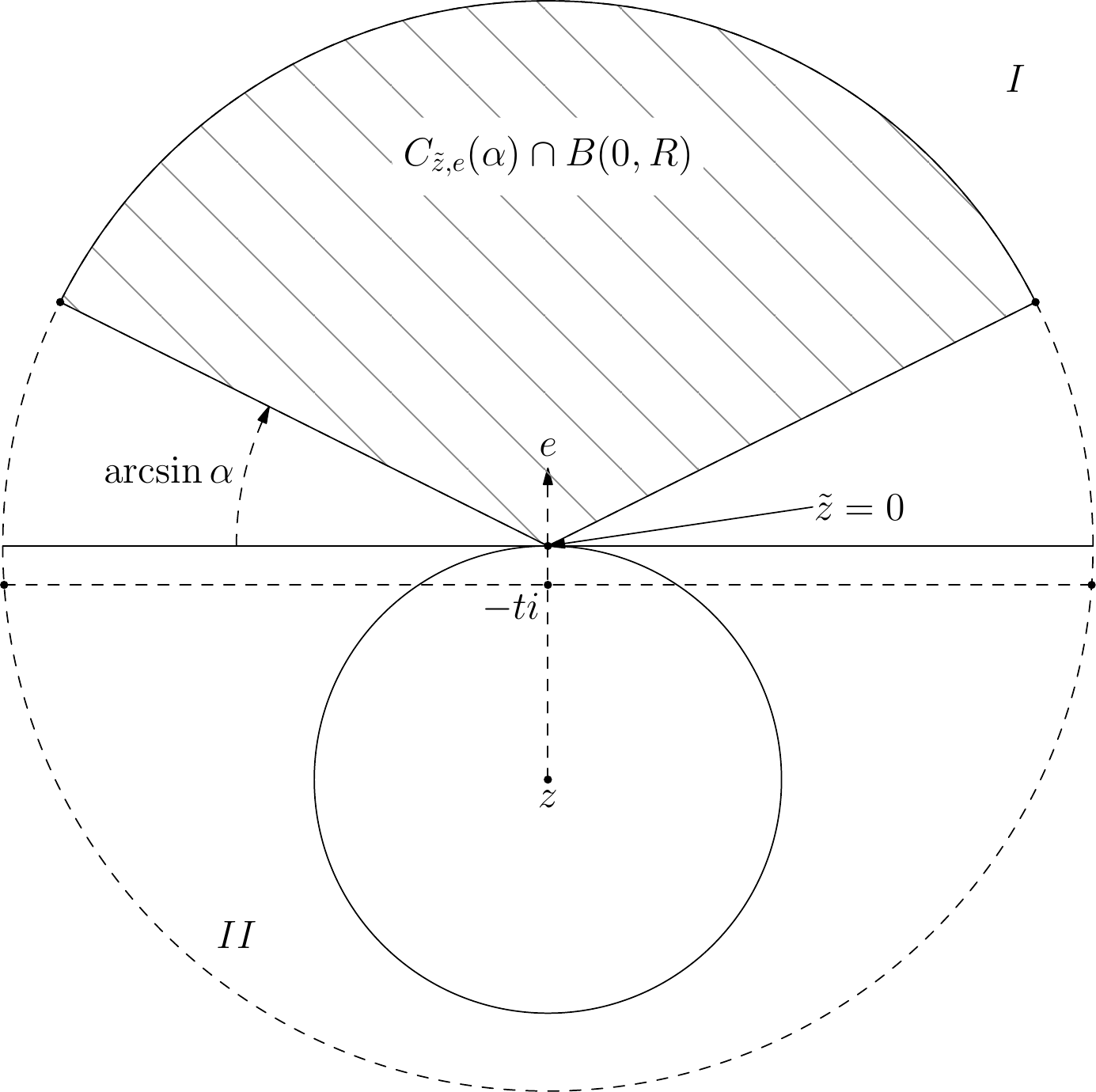}
\caption{The set-up for Lemma \ref{nothingotherside}}
 \end{figure}

 We may suppose that $z=-r i$ for some $r>0$ and $\tilde{z}=0$, (and so $e=i$).  We shall examine the imaginary part of the Cauchy transform evaluated at $-ti$ for  $t\in (0,\tfrac{r}{2})$:
 $$\Im [\widetilde{\Cau}_{\mu}(1)(-ti)] = \int_{\mathbb{C}}\Bigl[\frac{\Im(\xi)+t}{|\xi+it|^2}-\frac{\Im(\xi-z_0)}{|\xi-z_0|^2}\Bigl]d\mu(\xi).
 $$

Lemma \ref{resolvlem} guarantees that $\Im [\widetilde{\Cau}_{\mu}(1)(-ti)] = \Im[\widetilde{\Cau}_{\mu}(1)(z)]$ for any $t>0$.  In particular, it is bounded independently of $t$.

Making reference to Figure 5, we let $R>3r$, and define three regions:
$I = \mathbb{C}\backslash B(0,R)$,  $II = \bigl\{\xi\in B(0,R) : \Im(\xi)<-t\}$, and $III = B(0,R)\backslash II.$

Set $d_0=\text{dist}(z_0, \supp(\mu)).$  First note that if $R>3|z_0|$, and $\xi \in \mathbb{C}\backslash B(0,R)$, then
$$\Bigl| \frac{\Im(\xi)+t}{|\xi+it|^2}-\frac{\Im(\xi-z_0)}{|\xi-z_0|^2}\Bigl|\leq \frac{C}{|\xi|^2}.
$$
Therefore,
\begin{equation}\begin{split}\nonumber \int_{|\xi|\geq R}&\Bigl|\frac{\Im(\xi)+t}{|\xi+it|^2} - \frac{\Im(\xi-z_0)}{|\xi-z_0|^2}\Bigl|d\mu(\xi)+ \int_{B(0,R)}\Bigl|\frac{\Im(\xi-z_0)}{|\xi-z_0|^2}\Bigl|d\mu(\xi)\\
&\leq  \int_{|\xi|\geq R} \frac{C}{|\xi|^2}d\mu(\xi)+ \int_{B(0, R)}\frac{1}{|\xi-z_0|}d\mu(\xi).
\end{split}\end{equation}
The right hand side of this inequality if finite and independent of $t$.

Next, note that if $\xi\in II\cap \supp(\mu)$, then $|\xi-it|^2\geq -(\Im(\xi)+t)r$, provided that $|\Im(\xi)|<\tfrac{r}{2}$ and $t<\tfrac{r}{2}$.  To see this, note that $|\xi-z|>r$, and so by elementary geometry, $|\xi-it|^2\geq r^2 - (r+(\Im (\xi)+t))^2$. This is at least $-(\Im(\xi)+t)r$ under our assumptions on $\xi$ and $t$.   Therefore, if $t<\tfrac{r}{2}$, then
$$\int_{II}\frac{\Im(\xi)+t}{|\xi+it|^2}d\mu(\xi)\geq -\int_{II\cap B(0, \tfrac{r}{2})}\frac{1}{r}d\mu(\xi) - \Bigl|\int_{II\backslash B(0,\tfrac{r}{2})}\frac{\Im(\xi)+t}{|\xi+it|^2}d\mu(\xi)\Bigl|.
$$
Both terms on the right hand side are bounded in absolute value by $C\tfrac{\mu(B(0,R))}{r}\leq \tfrac{CR}{r}$  (recall that $B(z,r)\cap\supp(\mu)=\varnothing$).  Note that the integral on the left hand side is at most zero.

Our conclusion thus far is that there is a constant $\Delta$, depending on $C_0,d_0, R,r$ and $\Im(\widetilde{\Cau}_{\mu}(1)(z))$, such that for any $t<\tfrac{r}{2}$,
\begin{equation}\label{IIIbound}\Bigl|\int_{III}\frac{\Im(\xi)+t}{|\xi+it|^2}d\mu(\xi)\Bigl|\leq \Delta.
\end{equation}
Note that the integrand in this integral is positive for any $\xi\in III$.

Suppose now that the statement of the lemma is false.  Then there exists $\alpha>0$, along with a sequence $z_j\in C_{0,e}(\alpha)\cap \supp(\mu)$ with $z_j\rightarrow 0$ as $j\rightarrow \infty$.  By passing to a subsequence, we may assume that $|z_j|\leq\tfrac{R}{2}$ for each $j$, and also that the balls $B_j = B(z_j, \tfrac{\alpha}{2}|z_j|)$ are pairwise disjoint.

Each ball $B_j\subset III$, and provided that $t\leq \tfrac{\alpha}{2}|z_j|$, we have
$$\frac{\Im(\xi)+t}{|\xi+ti|^2}\geq \frac{\alpha|z_j|}{8|z_j|^2} = \frac{\alpha}{8|z_j|},  \text{ for any }\xi \in B_j.
$$
As a result, we see that
$$\int_{III}\frac{\Im(\xi)+t}{|\xi+it|^2}d\mu(\xi)\geq \sum_{j: \, t\leq|z_j|/2} \int_{B_j}\frac{\Im(\xi)+t}{|\xi+ti|^2}d\mu(\xi) \geq \sum_{j: \, t\leq|z_j|/2}\mu(B_j)\frac{\alpha}{8|z_j|}.
$$
But $\mu(B_j)\geq \tfrac{c_0\alpha}{2}|z_j|$, and so the previous integral over $III$ has size at least $\tfrac{c_0\alpha^2}{16}\cdot\text{card}\{j: t\leq \tfrac{1}{2}|z_j|\}$. However, if $t$ is sufficiently small, then this quantity may be made larger than the constant $\Delta$ appearing in (\ref{IIIbound}).  This is absurd.
\end{proof}

We now pause to prove a simple convergence lemma.

\begin{lem}\label{outsideconv}Suppose that $\nu_k$ is a sequence of $C_0$-nice measures with AD-regularity constant $c_0$ that converges to $\nu$ weakly as $k\rightarrow \infty$ (and so $\nu$ is $C_0$-nice with AD-regularity constant $c_0$).  If $z_0\not\in\supp(\nu)$, then for any $z\not\in \supp(\nu)$, $\widetilde{\Cau}_{\nu_k}(1)(z)$ is well defined (as in (\ref{z0defn})) for sufficiently large $k$, and $\widetilde{\Cau}_{\nu_k}(1)(z)\rightarrow \widetilde{\Cau}_{\nu}(1)(z)$ as $k\rightarrow \infty$.\end{lem}

\begin{proof}  First note that there exists $r>0$ such that $\nu(B(z_0,r))=0=\nu(B(z,r))$.  But then, by the AD regularity of each $\nu_k$, we must have that $\nu_k(B(z_0, \tfrac{r}{2}))=0=\nu_k(B(z,\tfrac{r}{2}))$ for sufficiently large $k$, and hence $\widetilde{\Cau}_{\nu_k}(1)(z)$ is well defined.  Let $N>0$, and choose $\varphi_N \in \Lip_0(\mathbb{C})$ satisfying $\varphi_N\equiv 1$ on $B(0,N)$ and $0\leq \varphi_N\leq 1$ in $\mathbb{C}$.  For large enough $k$, $|\widetilde\Cau_{\nu}(1)(z) -\widetilde\Cau_{\nu_k}(1)(z)|$ is no greater than the sum of $|\int_{\mathbb{C}}[K(z-\xi)-K(z_0-\xi)]\varphi_N(\xi) d(\nu-\nu_k)(\xi)|$ and $|\int_{\mathbb{C}}[K(z-\xi)-K(z_0-\xi)][1-\varphi_N(\xi)] d(\nu-\nu_k)(\xi)|$.  The first of these two terms tends to zero as $k\rightarrow\infty$, while the second has size at most $\tfrac{C|z-z_0|}{N}$ (for sufficiently large $N$) due to Lemma \ref{l1awayfromsupp}.  This establishes the required convergence.\end{proof}

\begin{lem}\label{halfspacelem}  Suppose that $z\not\in \supp(\mu)$, and $\tilde z$ is a closest point on the support of $\mu$ to $z$.  Let $e= \tfrac{\tilde{z}-z}{|\tilde{z}-z|}$.  Then $\supp(\mu)\subset H_{\tilde{z},e}$.
\end{lem}

\begin{proof}  Write $e=e^{i\theta}$.  By translation, we may assume that $\tilde{z}=0$.  To prove the lemma, it suffices to show that $B(-\rho e, \rho)\cap \supp(\mu)=\varnothing$ for all $\rho>0$.

Fix $t_0$ small enough to ensure that $te\not\in\supp(\mu)$ for any $0<t\leq t_0$.  The existence of $t_0>0$ is guaranteed by Lemma \ref{nothingotherside}.  Now set $\sigma = \widetilde{\Cau}_{\mu}(1)(z)-\widetilde{\Cau}_{\mu}(1)(t_0e)$.  Notice that the value of $\sigma$ is independent of the choice of $z_0\not\in \supp(\mu)$ in (\ref{z0defn}), so we shall fix $z_0=t_0e$.  Now, let $\mu^{\star}$ denote a tangent measure to $\mu$ at $0$.  On account of Lemma \ref{nothingotherside}, the support of $\mu^{\star}$ is contained in the line $L$ through $0$ perpendicular to $e$.  By Step 1, $\mu^{\star} = c^{\star}\mathcal{H}^1_{L}$ with $c^{\star}\in [c_0,C_0]$.  As a result, for any $y>0$, we have that
\begin{equation}\label{secondjump}\widetilde{\Cau}_{\mu^{\star}}(1)(\!-ye)\! -\!\widetilde{\Cau}_{\mu^{\star}}(1)(ye)  \!= \! \int_{\mathbf{R}}\! \frac{-2e^{-i\theta}y}{t^2+y^2} c^{\star} dm_1(t)\!=\!-2\pi c^{\star} e^{-i\theta}.
\end{equation}
We claim that $\widetilde{\Cau}_{\mu^{\star}}(1)(-ye) - \widetilde{\Cau}_{\mu^{\star}}(1)(ye) = \sigma$. To see this, note that for $\lambda>0$ small enough so that $y\lambda\leq t_0$, we have $\widetilde{\Cau}_{\mu_{0,\lambda}}(1)(-ye)-\widetilde{\Cau}_{\mu_{0,\lambda}}(1)(ye) = \widetilde{\Cau}_{\mu}(1)(-\lambda ye)-\widetilde{\Cau}_{\mu}(1)(\lambda ye)$.  But this equals $\sigma$ because $-\lambda ye$ and $\lambda ye$ lie in the same connected components of $\mathbb{C}\backslash\supp(\mu)$ as $z$ and $t_0 e$ respectively.  Since $\mu^{\star}$ is a weak limit of measures $\mu_{0,\lambda_k}$ for some sequence $\lambda_k\rightarrow 0$, applying Lemma \ref{outsideconv} proves the claim.  Consequently, we have that $\sigma$ determines the direction of tangency from $z$ to $\supp(\mu)$ (the angle $\theta$).

The right hand side of (\ref{secondjump}) is non-zero, and so $t_0e$ lies in a different component of $\mathbb{C}\backslash\supp(\mu)$ to $z$.  As there are only two possible values that $\widetilde{\Cau}_{\mu}(1)$ can take in $\mathbb{C}\backslash\supp(\mu)$, $\sigma$ is determined by $\widetilde{\Cau}_{\mu}(1)(z)$.  Since $\widetilde{\Cau}_{\mu}(1)$ is constant in each connected component of $\mathbb{C}\backslash \supp(\mu)$, the direction of tangency from any point in the connected component of $\mathbb{C}\backslash \supp(\mu)$ containing $z$ to $\supp(\mu)$ is the same.
Finally, set \begin{equation}\begin{split}\nonumber\mathcal{I} = \bigl\{\rho>0: \{-t e: t\in (0,\rho]\}&\text{ lies in the same connected component}\\
&\text{ of }\mathbb{C}\backslash\supp(\mu)\text{ as }z\bigl\}.\end{split}\end{equation}  We claim that if $\rho\in \mathcal{I}$, then $B(-\rho e, \rho)\cap \supp(\mu)=\varnothing$.  Indeed, otherwise there is a point $\zeta\neq 0$ which is a closest point in $\supp(\mu)$ to $-\rho e$.  But then it follows that $e=\tfrac{\zeta+\rho e}{|\zeta+\rho e|}$.  Given that $\{-te: t\in (0, \rho]\}\cap \supp(\mu)=\varnothing$, this is a contradiction.   From this claim, we see that if $\rho \in \mathcal{I}$, then $(0,2\rho)\subset\mathcal{I}$.  Since $|z|\in \mathcal{I}$, it follows that $\mathcal{I}=(0,\infty)$, so $B(-\rho e, \rho)\cap \supp(\mu)=\varnothing$ for any $\rho>0$.
\end{proof}

\begin{proof}[Proof of Proposition \ref{reflectioncharac}]
An immediate corollary of Lemma \ref{halfspacelem} is the following statement:  For each $z\not\in \supp(\mu)$, there is a half space with $z$ on its boundary which does not intersect $\supp(\mu)$.

Now, suppose that there are three points $z, \xi,\zeta \in \text{supp}(\mu)$ which are not collinear.  Then they form a triangle.  Since $\mu$ is AD-regular, there is a point in the interior of this triangle outside of the support of $\mu$.  Let's call this point $\omega$.  But then there is a half space, with $\omega$ on its boundary, which is disjoint from $\supp(\mu)$. This half space must contain at least one of the points $z$, $\xi$ or $\zeta$.  This is absurd.
\end{proof}



 \end{document}